\documentclass{amsart}
\usepackage{graphicx}
\usepackage{amsmath}
\usepackage{amssymb}
\usepackage{amsthm}
\usepackage{esint}

\newtheorem{theorem}{Theorem}[section]
\newtheorem{lem}[theorem]{Lemma}

\theoremstyle{Corollary}
\newtheorem{cor}[theorem]{Corollary}
\newtheorem{prop}[theorem]{Proposition}

\numberwithin{equation}{section}

\begin{document}

\title{
First eigenvalues of geometric operators under the Yamabe flow}

\author{Pak Tung Ho}
\address{Department of Mathematics, Sogang University, Seoul
121-742, Korea}

\email{ptho@sogang.ac.kr, paktungho@yahoo.com.hk}


\date{August 12, 2016.}

\keywords{Yamabe flow; eigenvalue; CR manifold\\
This work was supported by the National Research Foundation of Korea (NRF) grant funded
by the Korea government (MEST) (No.201531021.01).}

\begin{abstract}
Suppose $(M,g_0)$ is a compact Riemannian manifold without boundary
of dimension $n\geq 3$.
Using the Yamabe flow, we obtain estimate for the first nonzero eigenvalue of
the Laplacian of $g_0$ with negative scalar curvature
in terms of the Yamabe metric in its conformal class.
On the other hand, we
prove that the first eigenvalue of some geometric operators
on a compact Riemannian manifold
is nondecreasing along the unnormalized Yamabe flow under suitable curvature assumption.
Similar results are obtained for manifolds with boundary and for CR manifold.
\end{abstract}

\maketitle

\section{Introduction}

Suppose $(M,g_0)$ is a compact $n$-dimensional manifold without boundary where $n\geq 3$.
As a generalization of Uniformization theorem, the Yamabe problem is
to find a metric $g$ conformal to $g_0$ such that its scalar curvature $R_g$ is constant.
The Yamabe problem was first studied by Yamabe in \cite{Yamabe}.
Note that if we write $g=u^{\frac{4}{n-2}}g_0$ where $u$ is a positive smooth function in $M$, then the scalar curvature $R_g$ of $g$ can be written
as
\begin{equation}\label{Yamabe}
R_g=u^{-\frac{n+2}{n-2}}\left(-\frac{4(n-1)}{n-2}\Delta_{g_0}u+R_{g_0}u\right).
\end{equation}
Therefore, the Yamabe problem is to solve (\ref{Yamabe}) such that $R_g$ is constant.
This was solved by Trudinger \cite{Trudinger}, Aubin \cite{Aubin0}, and
Schoen \cite{Schoen}. See the survey article \cite{Lee&Parker} of Lee and Parker
for more details.

Yamabe flow was introduced by Hamilton in \cite{Hamilton} to study the Yamabe problem, which is defined as the evolution of
the metric $g=g(t)$:
$$\frac{\partial}{\partial t}g=-(R_{g}-\overline{R}_{g})g\mbox{ for }t\geq 0,\hspace{2mm}g|_{t=0}=g_0,$$
where $\overline{R}_{g}=\displaystyle\frac{\int_MR_{g}dV_{g}}{\int_MdV_{g}}$ is the average of the scalar curvature
${R}_{g}$ of the Riemannian metric $g$.
In \cite{Chow}, Chow proved that the Yamabe flow approaches a metric of constant scalar
curvature provided that the initial metric is locally conformally flat and has positive
Ricci curvature. In \cite{Ye}, Ye proved the convergence of the Yamabe flow by assuming
only that the initial metric is locally conformally flat. Later, Schwetlick and Struwe
\cite{Schwetlick&Struwe} proved the convergence of the Yamabe flow for the case when $3\leq n\leq 5$
under the assumption that the initial metric has large energy. Finally Brendle \cite{Brendle4,Brendle5}
showed that the Yamabe flow converges to a metric of
constant scalar curvature by using positive mass theorem.
See also \cite{Barbosa,Daskalopoulos,Ma1,Ma2,Pablo} for results related to the Yamabe flow.

In this paper, we prove the following theorem by using the Yamabe flow.

\begin{theorem}\label{main}
Suppose $(M,g_0)$ is a compact Riemannian manifold of dimension $n\geq 3$ without boundary such that
$\displaystyle\max_MR_{g_0}< 0$,
and $g_Y$ is the Yamabe metric conformal to $g_0$
which has same volume as $g_0$.
Then the first nonzero eigenvalues of the Laplacian of $g_0$ and $g_Y$ satisfy
\begin{equation}\label{20}
e^{-c}\lambda_1(g_Y)\leq\lambda_1(g_0)\leq
e^{c}\lambda_1(g_Y)
\end{equation}
where $c=\displaystyle2(n-1)\left(\frac{\min_MR_{g_0}}{\max_MR_{g_0}}-1\right).$
\end{theorem}

Recall that $g_Y$ is a Yamabe metric in the conformal class of $g_0$ if $g_Y$ is a
Riemannian metric conformal to $g_0$ such that its scalar curvature is constant.
Theorem \ref{main} can be applied to estimate the
first nonzero eigenvalue of a metric with negative scalar curvature
in terms of the Yamabe metric in its conformal class. See Theorem \ref{main1}.

The Yamabe problem was also studied in the context of manifolds with boundary.
Suppose $(M,g_0)$ is a compact $n$-dimensional manifold with smooth boundary $\partial M$ where $n\geq 3$.
The Yamabe problem is
to find a metric $g$ conformal to $g_0$ such that its scalar curvature $R_g$ is constant in $M$
and its mean curvature $H_g$ is zero on $\partial M$.
This has been studied by Escobar in \cite{Escobar1}.
See also \cite{Brendle&Chen,Han&Li}
for results in this direction.
In particular, the Yamabe flow was introduced on manifolds with boundary by Brendle in \cite{Brendle1}:
given a metric $g_0$ with vanishing mean curvature on the boundary, i.e. $H_{g_0}=0$ on $\partial M$,
we can define the Yamabe flow as follows:
$$
\frac{\partial}{\partial t}g=-(R_g-\overline{R}_g)g\mbox{ in }M\mbox{ and }H_g=0 \hbox{ on }\partial M
\mbox{ for }t\geq 0,\hspace{2mm}g|_{t=0}=g_0.
$$
Using the Yamabe flow, we obtain estimate for
the first nonzero eigenvalue of
the Laplacian  with Dirichlet boundary condition
when $(M,g_0)$ has negative scalar curvature.
See Theorem \ref{mainB}.

One can consider the following CR analogue of the Yamabe problem,
the CR Yamabe problem. Suppose  $(M,\theta_0)$ is a compact strictly
pseudoconvex CR manifold of real dimension $2n+1$ with a contact
form $\theta_0$.
The CR Yamabe problem
is to find a contact form $\theta$ conformal to $\theta_0$ such that
its Webster scalar curvature is constant. Jerison and Lee
\cite{Jerison&Lee1,Jerison&Lee2,Jerison&Lee3}
solved the CR Yamabe problem when $n\geq 2$ and
$M$ is not locally CR equivalent to the sphere.
The remaining cases, namely when $n=1$ or $M$ is locally CR
equivalent to the sphere, were studied
respectively by Gamara and Yacoub in \cite{Gamara1} and by Gamara in
\cite{Gamara2}.
See also the recent work of Cheng-Chiu-Yang in \cite{Cheng&Chiu&Yang} and
Cheng-Malchiodi-Yang in \cite{Cheng&Malchiodi&Yang}.

The CR Yamabe flow was introduced to study the CR Yamabe problem, which is defined as:
$$
\frac{\partial}{\partial t}\theta=-(R_\theta-\overline{R}_\theta)\theta\mbox{ for }t\geq 0,\hspace{2mm}\theta|_{t=0}=\theta_0,
$$
where $R_\theta$ is the Webster scalar curvature of the contact form $\theta$, and $\overline{R}_\theta$ is the average of the Webster scalar curvature.
See \cite{Chang,Chang&Chiu&Cheng,Ho1,Ho2,Ho4,Ho5,Zhang}
for results related to the CR Yamabe flow.
Using the CR Yamabe flow, we obtain estimate for
the first nonzero eigenvalue of
the sub-Laplacian of a contact form $\theta_0$ with negative Webster scalar curvature.
See Theorem \ref{mainA}.

In another direction, we consider eigenvalues of some
geometric operators under the unnormalized Yamabe flow.
In recent years,
there has been increasing attentions on the study of eigenvalues of geometric operators under different kinds of geometric flow.
In \cite{Perelman}, Perelman proved that the first eigenvalue of $-4\Delta_g+R_g$
is nondecreasing along the  Ricci flow
$$\frac{\partial}{\partial t}g=-2Ric_g,$$
where $Ric_g$ and $R_g$ are the Ricci curvature and scalar curvature of $g$ respectively.
As an application, he showed that there is no nontrivial steady or expanding breathers on closed manifolds.
In \cite{Cao1}, Cao showed that the eigenvalues of $-\Delta_g+\displaystyle\frac{1}{2}R_g$ on
Riemannian manifolds with nonnegative curvature operator are nondecreasing under the Ricci flow.
See \cite{Li,Li1,Ma,Zhao}
for related results.

In \cite{Cao}, Cao proved that the first eigenvalue of $-\Delta_g+aR_g$ on a closed manifold $M$, where $a>1/4$, is nondecreasing along
the Ricci flow. In \cite{Cao&Hou&Ling}, Cao-Hou-Ling showed that the first eigenvalue of
$-\Delta_g+aR_g$, where $0<a<1/2$, on a closed surface with nonnegative scalar curvature is nondecreasing
under the Ricci flow.
Combining these results, we have the following: (see Theorem 2.2 in \cite{Cao&Hou&Ling})
\begin{theorem}\label{Mmain1}
On a closed surface with nonnegative scalar curvature, for all $a>0$, the first eigenvalue
of $-\Delta_g+aR_g$ is nondecreasing under the Ricci flow.
\end{theorem}

Note that when the dimension $n=2$, we have $Ric_g=\displaystyle\frac{1}{2}R_gg$. Therefore,
the Ricci flow becomes the unnormalized Yamabe flow
$$\frac{\partial}{\partial t}g=-R_g g.
$$
We would like to generalize Theorem \ref{Mmain1} to higher dimension by considering the  unnormalized Yamabe flow.
In particular, we prove the following:

\begin{theorem}\label{Mmain}
Along the unnormalized Yamabe flow,
the first eigenvalue of $-\Delta_g+aR_g$ is nondecreasing\\
(i)
if $0\leq a<\displaystyle\frac{n-2}{4(n-1)}$ and
$\displaystyle\min_MR_g\geq\frac{n-2}{n}\max_MR_g\geq 0$;\\
(ii) if $a\geq\displaystyle\frac{n-2}{4(n-1)}$ and
$\displaystyle\min_MR_g\geq 0$.
\end{theorem}

Corresponding results are also obtained for manifolds with boundary and for CR manifolds. See Theorem \ref{Mmain3},
Theorem \ref{Mmain4} and
Theorem \ref{Mmain2}. Note that Theorem \ref{Mmain} was proved in \cite{Ho2}
for the cases when $a=0$ and $a=\displaystyle \frac{n-2}{4(n-1)}$ (c.f. Theorem 6.1
and Theorem 6.2 in \cite{Ho2}).
Note also that the condition $\displaystyle\min_MR_g\geq 0$ is preserved by the Yamabe flow (see Proposition 6.2
in \cite{Ho2}). Therefore, Theorem \ref{Mmain} implies that following corollary, which
can be considered as a generalization of  Theorem \ref{Mmain1} in higher dimensions.

\begin{cor}
On a closed Riemannian manifold with nonnegative scalar curvature,
the first eigenvalue of $-\Delta_g+aR_g$, where $a\geq\displaystyle\frac{n-2}{4(n-1)}$, is nondecreasing
along the unnormalized Yamabe flow.
\end{cor}

We would like to point out the following main difference between the proof of Theorem \ref{main}
and Theorem \ref{Mmain}.
By the eigenvalue perturbation theory (c.f. \cite{Cao,Kato,Kleiner&Lott,Reed&Simon}), we know that
there is a family of
the first eigenvalue and its corresponding eigenfunction of the geometric operator,
which is $C^1$ in $t$ along the flow in
Theorem \ref{Mmain}.
However, in Theorem \ref{main},
we only know that the first nonzero eigenvalue (that is, the second eigenvalue since the
first eigenvalue is always zero) of the Laplacian
is Lipschitz continuous in $t$.
But we are able to overcome the difficulty by following the ideas
of Wu-Wang-Zheng in
\cite{Wu&Wang&Zheng}.

\section{The Yamabe flow on manifolds without boundary}

In this section, we let $(M,g_0)$ be a compact Riemannian manifold of dimension $n\geq 3$ without boundary.
We consider the Yamabe flow, which is defined as
\begin{equation}\label{1}
\frac{\partial}{\partial t}g=-(R_g-\overline{R}_g)g\mbox{ for }t\geq 0,\hspace{2mm}g|_{t=0}=g_0.
\end{equation}
Here $R_g$ is the scalar curvature of $g$ and $\overline{R}_g$ is the average of the scalar curvature
given by
\begin{equation}\label{0}
\overline{R}_g=\frac{\int_MR_gdV_g}{\int_MdV_g},
\end{equation}
where $dV_g$ is the volume form of $g$.
Since the Yamabe flow preserves the conformal structure, we can write $g=u^{\frac{4}{n-2}}g_0$
for some positive function $u$, where $u$ satisfies the following evolution equation:
\begin{equation}\label{2}
\frac{\partial u}{\partial t}=-\frac{n-2}{4}(R_g-\overline{R}_g)u\mbox{ for }t\geq 0,\hspace{2mm}u|_{t=0}=1.
\end{equation}
Hence, the volume form $dV_g$ of $g$ satisfies
\begin{equation}\label{3}
\frac{\partial}{\partial t}(dV_g)=\frac{\partial}{\partial t}(u^{\frac{2n}{n-2}}dV_{g_0})=
\frac{2n}{n-2}u^{\frac{2n}{n-2}-1}\frac{\partial u}{\partial t}dV_{g_0}=-\frac{n}{2}(R_g-\overline{R}_g)dV_g.
\end{equation}
On the other hand, the scalar curvature $R_g$ of $g$ satisfies the following evolution equation:
(see \cite{Brendle4})
\begin{equation}\label{4}
\frac{\partial}{\partial t}R_g=(n-1)\Delta_gR_g+R_g(R_g-\overline{R}_g).
\end{equation}

We have the following proposition, which is  inspired by  Proposition 3.1 in \cite{Wu&Wang&Zheng}.

\begin{prop}\label{AUX}
Let $g=g(t)$ be the solution of the
the Yamabe flow $(\ref{1})$
and $\lambda_1(t)$ be the corresponding first nonzero eigenvalue of the Laplacian.
Then for any $t_2\geq t_1$,
there exists a $C^\infty$ function $f$ on $M\times [t_1,t_2]$
satisfying
\begin{equation}\label{6}
\int_{M}f^2dV_{g}=1\hspace{2mm}\mbox{
and }\hspace{2mm}\displaystyle\int_{M}fdV_{g}=0\mbox{ for all }t\in [t_1,t_2],
\end{equation}
and
$$\lambda_1(t_2)\geq \lambda_1(t_1)
-\frac{n-2}{2}\int_{t_1}^{t_2}\int_M(R_g-\overline{R}_g)|\nabla_{g} f|^2_{g}dV_{g}dt-2\int_{t_1}^{t_2}\int_{M}\frac{\partial f}{\partial t}\Delta_{g} fdV_{g}dt$$
such that at time $t_2$, $f(t_2)$ is the corresponding eigenfunction of
$\lambda_1(t_2)$.
\end{prop}
\begin{proof}
At time $t_2$, we  let $f_2 = f(t_2)$ be the eigenfunction for the
first nonzero eigenvalue
$\lambda_1(t_2)$ of $g(t_2)$.
We define the following smooth function on $M$:
$$h(t)=\left(\frac{u(t_2)}{u(t)}\right)^{\frac{2n}{n-2}}f_2$$
where $u(t)$ is the solution of
(\ref{2}).
We normalize this smooth function on $M$ by
$$f(t)=\frac{h(t)}{(\int_{M}h(t)^2dV_{g(t)})^{\frac{1}{2}}}.$$
Then we can easily check that $f(t)$ satisfies (\ref{6}).

Set
$$G(g(t),f(t)):=\int_M|\nabla_{g(t)}f(t)|^2dV_{g(t)}.$$
Note that $G(g(t),f(t))$ is a smooth function in $t$.
Since $g=u^{\frac{4}{n-2}}g_0$, we have
\begin{equation}\label{9}
\langle\nabla_g f_1,\nabla_g f_2\rangle_g=u^{-\frac{4}{n-2}}\langle\nabla_{g_0} f_1,\nabla_{g_0} f_2\rangle_{g_0}
\end{equation}
for any functions $f_1, f_2$ in $M$, which implies that
\begin{equation*}
G(g(t),f(t))=\int_Mu^2|\nabla_{g_0} f|^2_{g_0}dV_{g_0}.
\end{equation*}
Differentiating it with respect to $t$, we get
\begin{equation}\label{3.13A}
\begin{split}
\mathcal{G}(g(t),f(t))&:=\frac{d}{dt}G(g(t),f(t))\\
&=\int_M2u\frac{\partial u}{\partial t}|\nabla_{g_0} f|^2_{g_0}dV_{g_0}
+2\int_Mu^2\langle\nabla_{g_0} f,\nabla_{g_0}(\frac{\partial f}{\partial t})\rangle_{g_0}dV_{g_0}\\
&=-\frac{n-2}{2}\int_M(R_g-\overline{R}_g)|\nabla_{g} f|^2_{g}dV_{g}
+2\int_M\langle\nabla_{g} f,\nabla_{g}(\frac{\partial f}{\partial t})\rangle_{g}dV_{g}\\
&=-\frac{n-2}{2}\int_M(R_g-\overline{R}_g)|\nabla_{g} f|^2_{g}dV_{g}
-2\int_M\frac{\partial f}{\partial t}\Delta_{g} fdV_{g}
\end{split}
\end{equation}
where the second last equality follows from (\ref{2}) and (\ref{9}),
the last equality follows from integration by parts.
It follows from the definition of $\mathcal{G}(g(t),f(t))$
in (\ref{3.13A}) that
\begin{equation}\label{3.13B}
G(g(t_2),f(t_2))-G(g(t_1),f(t_1))=\int_{t_1}^{t_2}\mathcal{G}(g(t),f(t))dt.
\end{equation}
Since $f(t_2)$ is the corresponding eigenfunction of
$\lambda_1(t_2)$, we have
\begin{equation}
G(g(t_2),f(t_2))=\lambda_1(t_2)\int_{M}f(t_2)^2dV_{g(t_2)}=\lambda_1(t_2)
\end{equation}
by (\ref{6}). On the other hand, it follows from (\ref{6}) and the definition
of $G(g(t),f(t))$ that
\begin{equation}\label{3.17A}
G(g(t_1),f(t_1))\geq \lambda_1(t_1)\int_{\partial M}f(t_1)^2dV_{g(t_1)}=\lambda_1(t_1).
\end{equation}
Now the result follows from (\ref{3.13A})-(\ref{3.17A}).
\end{proof}

Hence, we have the following:

\begin{prop}\label{prop2}
The first nonzero eigenvalue $\lambda_1$ of the Laplacian along the Yamabe flow
$(\ref{1})$ satisfies
$$\frac{d}{dt}\log\lambda_1\geq-\frac{n-2}{2}\Big(\max_MR_g-\overline{R}_g\Big)
+\frac{n}{2}\Big(\min_MR_g-\overline{R}_g\Big).$$
Here the derivative on the left hand side
is in the
sense of the $\liminf$ of backward difference quotients.
\end{prop}
\begin{proof}
Differentiate the first equation in (\ref{6}) with respect to $t$, we have
\begin{equation}\label{7}
\int_Mf\frac{\partial f}{\partial t}dV_g=\frac{n}{4}\int_Mf^2(R_g-\overline{R}_g)dV_g
\end{equation}
by (\ref{3}). On the other hand,
since $f(t_2)$ is the corresponding eigenfunction of
$\lambda_1(t_2)$, we have
\begin{equation}\label{8}
-\int_{M}\frac{\partial f(t_2)}{\partial t}\Delta_{g(t_2)} f(t_2)dV_{g(t_2)}
=\lambda_1(t_2)\int_Mf(t_2)\frac{\partial f(t_2)}{\partial t}dV_{g(t_2)}.
\end{equation}
Combining (\ref{7}) and (\ref{8}), we get
\begin{equation}\label{furthermore}
\begin{split}
-\int_{M}\frac{\partial f(t_2)}{\partial t}\Delta_{g(t_2)} f(t_2)dV_{g(t_2)}
&=\frac{n}{4}\lambda_1(t_2)\int_Mf(t_2)^2(R_{g(t_2)}-\overline{R}_{g(t_2)})dV_{g(t_2)}\\
&\geq\frac{n}{4}\lambda_1(t_2)\left(\min_MR_{g(t_2)}-\overline{R}_{g(t_2)}\right),
\end{split}
\end{equation}
where we have used (\ref{6}) in the last inequality.
Since $\lambda(t_2)$ is positive, we have for any $\epsilon>0$
\begin{equation}\label{further}
-\int_{M}\frac{\partial f(t_2)}{\partial t}\Delta_{g(t_2)} f(t_2)dV_{g(t_2)}
>\frac{n}{4}\lambda_1(t_2)\left(\min_MR_{g(t_2)}-\overline{R}_{g(t_2)}-\epsilon\right),
\end{equation}
by (\ref{furthermore}).
By the definition of $f$,
the function $s\mapsto\displaystyle\int_{M}\frac{\partial f(s)}{\partial t}\Delta_{g(s)} f(s)dV_{g(s)}$ is
continuous in $s$.
On the other hand,
the function $s\mapsto\displaystyle\left(\min_MR_{g(s)}-\overline{R}_{g(s)}\right)$ is
continuous in $s$.
Hence, it follows from (\ref{further})
that for any $\epsilon>0$
\begin{equation}\label{10}
-\int_{M}\frac{\partial f}{\partial t}\Delta_{g} fdV_{g}
\geq\frac{n}{4}\lambda_1(t_2)\left(\min_MR_{g}-\overline{R}_{g}-\epsilon\right)
\end{equation}
when $t$ is sufficiently closed to $t_2$.
On the other hand, we have
\begin{equation}\label{furthermore2}
\begin{split}
&-\int_M(R_{g(t_2)}-\overline{R}_{g(t_2)})|\nabla_{g(t_2)} f|^2_{g(t_2)}dV_{g(t_2)}\\
&\geq-\Big(\max_MR_{g(t_2)}-\overline{R}_{g(t_2)}\Big)\int_M|\nabla_{g(t_2)} f|^2_{g(t_2)}dV_{g(t_2)}
=-\Big(\max_MR_{g(t_2)}-\overline{R}_{g(t_2)}\Big)\lambda_1(t_2)
\end{split}
\end{equation}
by (\ref{6}) and the fact that $f(t_2)$ is the corresponding eigenfunction of
$\lambda_1(t_2)$.
Since $\lambda(t_2)$ is positive, we have for any $\epsilon>0$
\begin{equation}\label{further2}
-\int_M(R_{g(t_2)}-\overline{R}_{g(t_2)})|\nabla_{g(t_2)} f|^2_{g(t_2)}dV_{g(t_2)}>
-\Big(\max_MR_{g(t_2)}-\overline{R}_{g(t_2)}+\epsilon\Big)\lambda_1(t_2)
\end{equation}
By the definition of $f$,
the function $s\mapsto\displaystyle\int_M(R_{g(s)}-\overline{R}_{g(s)})|\nabla_{g(s)} f|^2_{g(s)}dV_{g(s)}$ is
continuous in $s$.
On the other hand,
the function $s\mapsto\displaystyle\left(\max_MR_{g(s)}-\overline{R}_{g(s)}\right)$ is continuous in $s$.
Hence, it follows from (\ref{further2})
that  for any $\epsilon>0$
\begin{equation}\label{11}
\begin{split}
-\int_M(R_g-\overline{R}_g)|\nabla_{g} f|^2_{g}dV_{g}
&\geq-\lambda_1(t_2)\Big(\max_MR_g-\overline{R}_g+\epsilon\Big)
\end{split}
\end{equation}
when $t$ is sufficiently closed to $t_2$.
Substituting (\ref{10}) and (\ref{11}) into
the inequality in Proposition \ref{AUX}, we obtain
\begin{equation*}
\begin{split}
&\lambda_1(t_2)-\lambda_1(t_1)\\
&\geq
-\frac{n-2}{2}\lambda_1(t_2)\int_{t_1}^{t_2}\Big(\max_MR_g-\overline{R}_g+\epsilon\Big)dt+\frac{n}{2}\lambda_1(t_2)
\int_{t_1}^{t_2}\Big(\min_MR_g-\overline{R}_g-\epsilon\Big)dt
\end{split}
\end{equation*}
for $t_1<t_2$ and $t_1$ sufficiently closed to $t_2$.
Dividing $t_2-t_1$
in the above inequality and letting $t_1$ go to $t_2$,
we obtain
\begin{equation}\label{2.16}
\begin{split}
&\liminf_{t_1\to t_2}\frac{\lambda_1(t_2)-\lambda_1(t_1)}{t_2-t_1}\\
&\geq
-\frac{n-2}{2}\lambda_1(t_2)\Big(\max_MR_{g(t_2)}-\overline{R}_{g(t_2)}+\epsilon\Big)+\frac{n}{2}\lambda_1(t_2)
\Big(\min_MR_{g(t_2)}-\overline{R}_{g(t_2)}-\epsilon\Big).
\end{split}
\end{equation}
Note that
\begin{equation}\label{2.16.5}
\liminf_{t_1\to t_2}\frac{\log\lambda_1(t_2)-\log\lambda_1(t_1)}{t_2-t_1}\geq\frac{1}{\lambda_1(t_2)}\liminf_{t_1\to t_2}\frac{\lambda_1(t_2)-\lambda_1(t_1)}{t_2-t_1}
\mbox{
for }t_1<t_2.
\end{equation}
To see this, note that
\begin{equation*}
\log\lambda_1(t_2)-\log\lambda_1(t_1)=\log\left(\Big(\frac{\lambda_1(t_2)}{\lambda_1(t_1)}-1\Big)+1\right)
=\Big(\frac{\lambda_1(t_2)}{\lambda_1(t_1)}-1\Big)+O\left(\Big(\frac{\lambda_1(t_2)}{\lambda_1(t_1)}-1\Big)^2\right)
\end{equation*}
which implies that
\begin{equation*}
\begin{split}
&\liminf_{t_1\to t_2}\frac{\log\lambda_1(t_2)-\log\lambda_1(t_1)}{t_2-t_1}\\
&\geq\liminf_{t_1\to t_2}\frac{1}{\lambda_1(t_1)}\left(\frac{\lambda_1(t_2)-\lambda_1(t_1)}{t_2-t_1}\right)
+\liminf_{t_1\to t_2}O\left(\frac{1}{t_2-t_1}\Big(\frac{\lambda_1(t_2)}{\lambda_1(t_1)}-1\Big)^2\right).
\end{split}
\end{equation*}
Note that
\begin{equation*}
\begin{split}
\liminf_{t_1\to t_2}\frac{1}{\lambda_1(t_1)}\left(\frac{\lambda_1(t_2)-\lambda_1(t_1)}{t_2-t_1}\right)
&\geq \left(\liminf_{t_1\to t_2}\frac{1}{\lambda_1(t_1)}\right)\left(\liminf_{t_1\to t_2}\frac{\lambda_1(t_2)-\lambda_1(t_1)}{t_2-t_1}\right)\\
&=\frac{1}{\lambda_1(t_2)}\liminf_{t_1\to t_2}\frac{\lambda_1(t_2)-\lambda_1(t_1)}{t_2-t_1}
\end{split}
\end{equation*}
and
$$\lim_{t_1\to t_2}\frac{1}{t_2-t_1}\Big(\frac{\lambda_1(t_2)}{\lambda_1(t_1)}-1\Big)^2=0$$
since $\lambda_1(t)$ is Lipschitz continuous in $t$,
(\ref{2.16.5}) follows from combining all these.
Now, combining (\ref{2.16}) and (\ref{2.16.5}), we have
\begin{equation*}
\begin{split}
&\liminf_{t_1\to t_2}\frac{\log\lambda_1(t_2)-\log\lambda_1(t_1)}{t_2-t_1}\\
&\geq -\frac{n-2}{2}\Big(\max_MR_{g(t_2)}-\overline{R}_{g(t_2)}+\epsilon\Big)
+\frac{n}{2}\Big(\min_MR_{g(t_2)}-\overline{R}_{g(t_2)}-\epsilon\Big).
\end{split}
\end{equation*}
Since $\epsilon>0$ is arbitrary, Proposition \ref{prop2} follows from letting $\epsilon\to 0$.
\end{proof}

Similarly, we can prove the following:

\begin{prop}\label{AUX1}
Let $g=g(t)$ be the solution of the
the Yamabe flow $(\ref{1})$
and $\lambda_1(t)$ be the corresponding first nonzero eigenvalue of the Laplacian.
Then for any $t_2\geq t_1$,
there exists a $C^\infty$ function $f$ on $M\times [t_1,t_2]$
satisfying
\begin{equation}\label{61}
\int_{M}f^2dV_{g}=1\hspace{2mm}\mbox{
and }\hspace{2mm}\displaystyle\int_{M}fdV_{g}=0\mbox{ for all }t\in[t_1,t_2],
\end{equation}
and
$$\lambda_1(t_2)\leq \lambda_1(t_1)
-\frac{n-2}{2}\int_{t_1}^{t_2}\int_M(R_g-\overline{R}_g)|\nabla_{g} f|^2_{g}dV_{g}dt-2\int_{t_1}^{t_2}\int_{M}\frac{\partial f}{\partial t}\Delta_{g} fdV_{g}dt$$
such that at time $t_1$, $f(t_1)$ is the corresponding eigenfunction of
$\lambda_1(t_1)$.
\end{prop}
\begin{proof}
We only sketch the proof since it is almost the same as the proof of Proposition \ref{AUX}.
At time $t_1$, we  let $f_1 = f(t_1)$ be the eigenfunction for the
first nonzero eigenvalue
$\lambda_1(t_1)$ of $g(t_1)$. We define
$$h(t)=\left(\frac{u(t_1)}{u(t)}\right)^{\frac{2n}{n-2}}f_1$$
where $u(t)$ is the solution of
(\ref{2}). Then
$$f(t)=\frac{h(t)}{(\int_{M}h(t)^2dV_{g(t)})^{\frac{1}{2}}}$$
satisfies (\ref{61}).
As in the proof of  Proposition \ref{AUX}, we define
$$G(g(t),f(t)):=\int_M|\nabla_{g(t)}f(t)|^2dV_{g(t)}.$$
Then it is clear that (\ref{3.13A}) and (\ref{3.13B}) are still true.
Since $f(t_1)$ is the corresponding eigenfunction of
$\lambda_1(t_1)$, we have
\begin{equation*}
G(g(t_1),f(t_1))=\lambda_1(t_1)\int_{M}f(t_1)^2dV_{g(t_1)}=\lambda_1(t_1)
\end{equation*}
by (\ref{61}). On the other hand, it follows from (\ref{61}) and the definition
of $G(g(t),f(t))$ that
\begin{equation*}
G(g(t_2),f(t_2))\geq \lambda_1(t_2)\int_{\partial M}f(t_2)^2dV_{g(t_2)}=\lambda_1(t_2).
\end{equation*}
Now Lemma \ref{AUX1} follows from combining all these.
\end{proof}

We also have the following:

\begin{prop}\label{prop21}
The first nonzero eigenvalue $\lambda_1$ of the Laplacian along the Yamabe flow
$(\ref{1})$ satisfies
$$\frac{d}{dt}\log\lambda_1\leq-\frac{n-2}{2}\Big(\min_MR_g-\overline{R}_g\Big)
+\frac{n}{2}\Big(\max_MR_g-\overline{R}_g\Big).$$
Here the derivative on the left hand side
is in the
sense of the $\limsup$ of forward difference quotients.
\end{prop}
\begin{proof}
Again we only sketch the proof since it is almost the same as the proof of Proposition \ref{prop2}.
As in the proof of Proposition \ref{prop2}, by (\ref{61}) and the fact that
$f(t_1)$ is the corresponding eigenfunction of
$\lambda_1(t_1)$, we have
\begin{equation*}
\begin{split}
-\int_{M}\frac{\partial f(t_1)}{\partial t}\Delta_{g(t_1)} f(t_1)dV_{g(t_1)}
&=\lambda_1(t_1)\int_Mf(t_2)\frac{\partial f(t_1)}{\partial t}dV_{g(t_1)}\\
&=\frac{n}{4}\lambda_1(t_1)\int_Mf(t_1)^2(R_{g(t_1)}-\overline{R}_{g(t_1)})dV_{g(t_1)}\\
&\leq\frac{n}{4}\lambda_1(t_1)\left(\max_MR_{g(t_1)}-\overline{R}_{g(t_1)}\right).
\end{split}
\end{equation*}
By continuity and the fact that $\lambda(t_1)>0$, we can conclude that for any $\epsilon>0$
\begin{equation}\label{2.17}
-\int_{M}\frac{\partial f}{\partial t}\Delta_{g} fdV_{g}
\leq\frac{n}{4}\lambda_1(t_1)\left(\max_MR_{g}-\overline{R}_{g}+\epsilon\right)
\end{equation}
when $t$ is sufficiently closed to $t_1$.
Similarly,
\begin{equation}\label{2.18}
\begin{split}
-\int_M(R_g-\overline{R}_g)|\nabla_{g} f|^2_{g}dV_{g}
&\leq-\lambda_1(t_1)\Big(\min_MR_g-\overline{R}_g-\epsilon\Big)
\end{split}
\end{equation}
when $t$ is sufficiently closed to $t_1$. Now
putting (\ref{2.17}) and (\ref{2.18}) into the inequality in Proposition \ref{AUX1}, we obtain
\begin{equation*}
\begin{split}
&\lambda_1(t_2)- \lambda_1(t_1)\\
&\leq
-\frac{n-2}{2}\lambda_1(t_1)\int_{t_1}^{t_2}\Big(\min_MR_g-\overline{R}_g-\epsilon\Big)dt+
\frac{n}{2}\lambda_1(t_1)
\int_{t_1}^{t_2}\Big(\max_MR_g-\overline{R}_g+\epsilon\Big)dt
\end{split}
\end{equation*}
for $t_2>t_1$ and $t_2$ sufficiently closed to $t_1$.
Dividing the last inequality by $t_2-t_1$ and
letting $t_2$ go to $t_1$, we get
\begin{equation}\label{2.21}
\begin{split}
&\limsup_{t_2\to t_1}\frac{\lambda_1(t_2)-\lambda_1(t_1)}{t_2-t_1}\\
&\leq
-\frac{n-2}{2}\lambda_1(t_1)\Big(\min_MR_{g(t_1)}-\overline{R}_{g(t_1)}-\epsilon\Big)+\frac{n}{2}\lambda_1(t_1)
\Big(\max_MR_{g(t_1)}-\overline{R}_{g(t_1)}+\epsilon\Big).
\end{split}
\end{equation}
Replacing $\displaystyle\liminf_{t_1\to t_2}$ by $\displaystyle\limsup_{t_2\to t_1}$
and reversing the sign of the inequality,
one can follow the arguments of proving (\ref{2.16.5}) to prove that
\begin{equation}\label{2.22}
\limsup_{t_2\to t_1}\frac{\log\lambda_1(t_2)-\log\lambda_1(t_1)}{t_2-t_1}\leq\frac{1}{\lambda_1(t_1)}\limsup_{t_2\to t_1}\frac{\lambda_1(t_2)-\lambda_1(t_1)}{t_2-t_1}
\mbox{
for }t_2>t_1.
\end{equation}
Combining (\ref{2.21}) and (\ref{2.22}), we have
\begin{equation*}
\begin{split}
&\limsup_{t_2\to t_1}\frac{\log\lambda_1(t_2)-\log\lambda_1(t_1)}{t_2-t_1}\\
&\leq-\frac{n-2}{2}\Big(\min_MR_{g(t_1)}-\overline{R}_{g(t_1)}-\epsilon\Big)+\frac{n}{2}
\Big(\max_MR_{g(t_1)}-\overline{R}_{g(t_1)}+\epsilon\Big).
\end{split}
\end{equation*}
Since $\epsilon>0$ is arbitrary, Proposition \ref{prop21} follows from letting $\epsilon\to 0$.
\end{proof}

Using the maximum principle, we can prove the following:

\begin{prop}\label{prop4}
If $\displaystyle\max_MR_{g_0}<0$, then
$\displaystyle\max_MR_{g}<0$ for all $t\geq 0$ under the Yamabe flow $(\ref{1})$.
\end{prop}
\begin{proof}
Consider the function $F$ on $M\times [0,\infty)$ given
by $$F(x,t)=R_g-\frac{1}{2}\max_MR_{g_0}.$$
We claim that $F<0$.
By contradiction, we suppose that
\begin{equation}\label{13}
F(x_0,t_0)\geq 0\mbox{ for some }(x_0,t_0)\in M\times [0,\infty).
\end{equation}
Since $g=g_0$ at $t=0$ and
$$F(x,0)=R_{g_0}-\frac{1}{2}\max_MR_{g_0}\leq \frac{1}{2}\max_MR_{g_0}<0$$ by assumption,
we must have $t_0>0$.
We assume that $t_0$ is the smallest time satisfying (\ref{13}), i.e.
\begin{equation}\label{14}
F(x,t)<0\mbox{ for all }t\in [0,t_0).
\end{equation}
By continuity, we have
\begin{equation}\label{14a}
F(x,t_0)\leq 0\mbox{ for all }x\in M.
\end{equation}
Combining (\ref{13}) and (\ref{14a}), we have
\begin{equation}\label{14b}
0=F(x_0,t_0)=\max_{x\in M}F(x,t_0),
\end{equation}
which implies
\begin{equation}\label{14c}
R_g(x_0,t_0)=\max_{x\in M}R_g(x,t_0)
\end{equation}
by the definition of $F$.

Therefore, at $(x_0,t_0)$, we have
\begin{equation}\label{14d}
\begin{split}
0\leq \frac{\partial F}{\partial t}=\frac{\partial R_g}{\partial t}
&=(n-1)\Delta_gR_g+R_g(R_g-\overline{R}_g)\\
&\leq R_g(R_g-\overline{R}_g),
\end{split}
\end{equation}
where the first inequality follows from
the fact that $F(x_0,t)$ is increasing at $t_0$
by (\ref{14})-(\ref{14b}), and the
second equality follows from (\ref{4}), and the last inequality follows from (\ref{14c}).
But this is a contradiction, since the last term of (\ref{14d}) is negative.
To see this, it follows from (\ref{14b}) and the definition of $F$ that
\begin{equation}\label{term1}
R_g(x_0,t_0)=\frac{1}{2}\max_MR_{g_0},
\end{equation}
which implies together with (\ref{14c})
that at $(x_0,t_0)$
\begin{equation}\label{term}
R_g(R_g-\overline{R}_g)=\frac{1}{2}\max_MR_{g_0}\left(R_g-\overline{R}_g\right)=\frac{1}{2}\max_MR_{g_0}\left(\max_{x\in M}R_g(x,t_0)-\overline{R}_g\right)
\end{equation}
Since
$\displaystyle\max_MR_{g_0}<0$ by assumption, (\ref{term}) is nonpositive and
is equal to zero if and only if
$\displaystyle\max_{x\in M}R_g(x,t_0)=\overline{R}_g,$
or equivalently, $g(t_0)$ has constant scalar curvature.
Hence, it follows from (\ref{term1}) that
\begin{equation*}
\overline{R}_{g(t_0)}=\frac{\int_{M}R_{g(t_0)}dV_{g(t_0)}}{\int_{M}dV_{g(t_0)}}=
\frac{1}{2}\max_MR_{g_0}>\max_MR_{g_0}\geq \overline{R}_{g_0},
\end{equation*}
which is a contradiction, since $t\mapsto \overline{R}_{g(t)}$ is nonincreasing
along the Yamabe flow (see (9) in \cite{Brendle4}).
This shows that (\ref{term}) must be negative, as we claimed.

This contradiction shows that (\ref{13}) is impossible. This proves that $F<0$, or equivalently,
$R_g<\frac{1}{2}\max_MR_{g_0}<0$, which proves the assertion.
\end{proof}

Similarly, we have the following:

\begin{prop}\label{prop5}
If $\displaystyle\max_MR_{g_0}<0$, then
$\displaystyle\max_MR_{g}\leq \max_MR_{g_0}$ for all $t\geq 0$ under the Yamabe flow $(\ref{1})$.
\end{prop}
\begin{proof}
The proof is similar to Proposition \ref{prop4}.
For $\epsilon>0$, we define the function $F(x,t)=R_g(x,t)-\epsilon (t+1)$ on $M\times [0,\infty)$.
We claim that $\displaystyle\max_M F<\max_MR_{g_0}$ on  $[0,\infty)$.
By contradiction, we suppose that
\begin{equation}\label{16}
F(x_0,t_0)\geq \max_MR_{g_0}\mbox{ for some }(x_0,t_0)\in M\times [0,\infty).
\end{equation}
Since $g=g_0$ at $t=0$, we have
$\displaystyle\max_MF=\max_MR_{g_0}-\epsilon<\max_MR_{g_0}$. Therefore, we must have $t_0>0$.
We may assume that $t_0$ is the smallest time satisfying (\ref{16}), i.e.
\begin{equation}\label{17}
F(x,t)<\max_MR_{g_0}\mbox{ for all }(x,t)\in M\times [0,t_0).
\end{equation}
By continuity, we have
\begin{equation}\label{17a}
F(x,t_0)\leq \max_MR_{g_0}\mbox{ for all }x\in M.
\end{equation}
Combining (\ref{16}) and (\ref{17a}), we have
\begin{equation}\label{17b}
F(x_0,t_0)=\max_{x\in M}F(x,t_0),
\end{equation}
which implies
\begin{equation}\label{17c}
R_g(x_0,t_0)=\max_{x\in M}R_g(x,t_0),
\end{equation}
by the definition of $F$.

Therefore, at $(x_0,t_0)$, we have
\begin{equation*}
\begin{split}
0\leq\frac{\partial F}{\partial t}=-\epsilon+\frac{\partial R_g}{\partial t}
&=-\epsilon+(n-1)\Delta_gR_g+R_g(R_g-\overline{R}_g)\\
&\leq -\epsilon+R_g(R_g-\overline{R}_g)\\
&=-\epsilon+\max_{x\in M}R_g(x,t_0)\left(\max_{x\in M}R_g(x,t_0)-\overline{R}_g\right)\leq -\epsilon,
\end{split}
\end{equation*}
where the first inequality follows from (\ref{17})-(\ref{17b}), the
second equality follows from (\ref{4}),
the second inequality and the third equality follows from (\ref{17c}),
and the last inequality follows from
the fact that $\displaystyle\max_MR_g< 0$ by Proposition \ref{prop4}.
This contradicts the assumption that $\epsilon>0$, which proves the claim.

By the claim, for any $\epsilon>0$, we have $\displaystyle\max_MF<\max_MR_{g_0}$ on $[0,\infty)$.
By letting $\epsilon\to 0$, we get the required estimate.
\end{proof}

We can also prove the following:

\begin{prop}\label{prop6}
If $\displaystyle\max_MR_{g_0}<0$, then
$\displaystyle\min_MR_{g}\geq \min_MR_{g_0}$ for all $t\geq 0$ under the Yamabe flow $(\ref{1})$.
\end{prop}
\begin{proof}
For $\epsilon>0$, we define the function $F(x,t)=R_g(x,t)+\epsilon (t+1)$ on $M\times [0,\infty)$.
We claim that $\displaystyle\min_M F>\min_MR_{g_0}$ on  $[0,\infty)$.
By contradiction, we suppose that
\begin{equation}\label{40}
F(x_0,t_0)\leq \max_MR_{g_0}\mbox{ for some }(x_0,t_0)\in M\times [0,\infty).
\end{equation}
Since $g=g_0$ at $t=0$, we have
$\displaystyle\min_MF=\min_MR_{g_0}+\epsilon>\min_MR_{g_0}$. Therefore, we must have $t_0>0$.
We may assume that $t_0$ is the smallest time satisfying (\ref{40}), i.e.
\begin{equation}\label{41}
F(x,t)>\min_MR_{g_0}\mbox{ for all }(x,t)\in M\times [0,t_0).
\end{equation}
By continuity, we have
\begin{equation}\label{42}
F(x,t_0)\geq \max_MR_{g_0}\mbox{ for all }x\in M.
\end{equation}
Combining (\ref{40}) and (\ref{42}), we have
\begin{equation}\label{43}
F(x_0,t_0)=\min_{x\in M}F(x,t_0),
\end{equation}
which implies
\begin{equation}\label{44}
R_g(x_0,t_0)=\min_{x\in M}R_g(x,t_0),
\end{equation}
by the definition of $F$.

Therefore, at $(x_0,t_0)$, we have
\begin{equation*}
\begin{split}
0\geq\frac{\partial F}{\partial t}=\epsilon+\frac{\partial R_g}{\partial t}
&=\epsilon+(n-1)\Delta_gR_g+R_g(R_g-\overline{R}_g)\\
&\geq \epsilon+R_g(R_g-\overline{R}_g)\\
&=\epsilon+\min_{x\in M}R_g(x,t_0)\left(\min_{x\in M}R_g(x,t_0)-\overline{R}_g\right)\geq \epsilon,
\end{split}
\end{equation*}
where the first inequality follows from (\ref{41})-(\ref{43}), the
second equality follows from (\ref{4}),
the second inequality and the third equality follows from (\ref{44}),
and the last inequality follows from
the fact that $\displaystyle\min_M R_g\leq\max_MR_g< 0$ by Proposition \ref{prop4}.
This contradicts the assumption that $\epsilon>0$, which proves the claim.

By the claim, for any $\epsilon>0$, we have $\displaystyle\min_MF>\min_MR_{g_0}$ on $[0,\infty)$.
By letting $\epsilon\to 0$, we get the required estimate.
\end{proof}

Using the maximum principle, we can also prove the following:

\begin{lem}\label{thm1}
If $\displaystyle\max_MR_{g_0}<0$, then
\begin{equation*}
R_{g(t)}\leq \overline{R}_{g_0}+\left(\max_MR_{g_0}-\min_MR_{g_0}\right)
+\left(\max_MR_{g_0}\right)\int_0^{t}\left(\max_MR_{g(s)}-\overline{R}_{g(s)}\right)ds.
\end{equation*}
 for all $t\geq 0$ under the Yamabe flow $(\ref{1})$.
\end{lem}
\begin{proof}
For $\epsilon>0$, we let
\begin{equation*}
\begin{split}
F(x,t)=&\,R_g(x,t)-\left(\max_MR_{g_0}-\min_MR_{g_0}\right)\\
&-\left(\max_MR_{g_0}\right)\int_0^{t}\left(\max_MR_{g(s)}-\overline{R}_{g(s)}\right)ds-\epsilon (t+1)
\end{split}
\end{equation*}
be a function defined on $M\times [0,\infty)$.
We claim that $F(x,t)<\overline{R}_{g_0}$.
By contradiction, we suppose that
\begin{equation}\label{18a}
F(x_0,t_0)\geq \overline{R}_{g_0}\mbox{ for some }(x_0,t_0)\in M\times [0,\infty).
\end{equation}
Since $g=g_0$ at $t=0$, we have
$$-\overline{R}_{g_0}+F(x,0)=-
\overline{R}_{g_0}+R_{g_0}-\left(\max_MR_{g_0}-\min_MR_{g_0}\right)-\epsilon\leq-\epsilon<0.$$ Therefore, we must have $t_0>0$.
We may assume that $t_0$ is the smallest time satisfying (\ref{18a}), i.e.
\begin{equation}\label{18b}
F(x,t)<\overline{R}_{g_0}\mbox{ for all }(x,t)\in M\times [0,t_0).
\end{equation}
By continuity, we have
\begin{equation}\label{18c}
F(x,t_0)\leq \overline{R}_{g_0}\mbox{ for all }x\in M.
\end{equation}
Combining (\ref{18a})-(\ref{18c}), we have
\begin{equation}\label{18d}
F(x_0,t_0)=\max_{x\in M}F(x,t_0),
\end{equation}
which implies
\begin{equation}\label{18e}
R_g(x_0,t_0)=\max_{x\in M}R_g(x,t_0)
\end{equation}
by the definition of $F$.

Therefore, at $(x_0,t_0)$, we have
\begin{equation*}
\begin{split}
0&\leq\frac{\partial F}{\partial t}=
\frac{\partial}{\partial t}R_g
-\left(\max_MR_{g_0}\right)\left(\max_MR_{g(t_0)}-\overline{R}_{g(t_0)}\right)-\epsilon\\
&=(n-1)\Delta_gR_g+R_g(R_g-\overline{R}_g)-\left(\max_MR_{g_0}\right)\left(\max_MR_{g(t_0)}-\overline{R}_{g(t_0)}\right)-\epsilon\\
&\leq R_g(R_g-\overline{R}_g)-\left(\max_MR_{g_0}\right)\left(\max_MR_{g(t_0)}-\overline{R}_{g(t_0)}\right)-\epsilon\\
&= \left(\max_MR_{g(t_0)}\right)\left(\max_MR_{g(t_0)}-\overline{R}_{g(t_0)}\right)-\left(\max_MR_{g_0}\right)\left(\max_MR_{g(t_0)}-\overline{R}_{g(t_0)}\right)-\epsilon\\
&\leq -\epsilon,
\end{split}
\end{equation*}
where the first inequality follows from (\ref{18b})-(\ref{18d}),
the second equality follows from (\ref{4}),  the second inequality
and the third equality follows from
(\ref{18e}), and the last inequality follows from Proposition \ref{prop5}.
This contradicts to the assumption that $\epsilon>0$.
This contradiction shows that $F(x,t)<\overline{R}_{g_0}$.
Letting $\epsilon\to 0$, we get the desired result.
\end{proof}

Similarly, we can prove the following:

\begin{lem}\label{thm2}
If $\displaystyle\max_MR_{g_0}<0$, then
\begin{equation*}
R_{g(t)}\geq\overline{R}_{g_0}-\left(\max_MR_{g_0}-\min_MR_{g_0}\right)
+\left(\max_MR_{g_0}\right)\int_0^{t}\left(\min_MR_{g(s)}-\overline{R}_{g(s)}\right)ds.
\end{equation*}
 for all $t\geq 0$ under the Yamabe flow $(\ref{1})$
\end{lem}
\begin{proof}
We only sketch the proof since it is essentially
the same as the proof of Lemma \ref{thm1}.
For $\epsilon>0$, we define the function
\begin{equation*}
\begin{split}
F(x,t)=&\,R_g(x,t)+\left(\max_MR_{g_0}-\min_MR_{g_0}\right)\\
&-\left(\max_MR_{g_0}\right)\int_0^{t}\left(\min_MR_{g(s)}-\overline{R}_{g(s)}\right)ds+\epsilon (t+1)
\end{split}
\end{equation*}
on $M\times [0,\infty)$.
We claim that $F(x,t)>\overline{R}_{g_0}$.
If it were not true, then we could find $(x_0,t_0)$ such that at $(x_0,t_0)$
\begin{equation*}
\begin{split}
0&\geq\frac{\partial F}{\partial t}=
\frac{\partial}{\partial t}R_g
-\left(\max_MR_{g_0}\right)\left(\min_MR_{g(t_0)}-\overline{R}_{g(t_0)}\right)+\epsilon\\
&=(n-1)\Delta_gR_g+R_g(R_g-\overline{R}_g)-\left(\max_MR_{g_0}\right)\left(\min_MR_{g(t_0)}-\overline{R}_{g(t_0)}\right)+\epsilon\\
&\geq R_g(R_g-\overline{R}_g)-\left(\max_MR_{g_0}\right)\left(\min_MR_{g(t_0)}-\overline{R}_{g(t_0)}\right)+\epsilon\\
&= \left(\min_MR_{g(t_0)}\right)\left(\min_MR_{g(t_0)}-\overline{R}_{g(t_0)}\right)-\left(\max_MR_{g_0}\right)\left(\min_MR_{g(t_0)}-\overline{R}_{g(t_0)}\right)+\epsilon\\
&\geq \epsilon,
\end{split}
\end{equation*}
where the last inequality follows from Proposition \ref{prop5}.
This contradicts to the assumption that $\epsilon>0$.
This contradiction shows that $F(x,t)>\overline{R}_{g_0}$.
Letting $\epsilon\to 0$, we get the desired result.
\end{proof}

Now we are ready to prove Theorem \ref{main}.

\begin{proof}[Proof of Theorem \ref{main}]
It was proved by Ye  (see Theorem 2 in \cite{Ye}) that $g\to g_\infty$ as $t\to\infty$
under the Yamabe flow (\ref{1}) such that $g_\infty$ is conformal to $g_0$ and has constant negative scalar curvature.
Along the Yamabe flow (\ref{1}), we have
$$\frac{d}{dt}\left(\int_MdV_g\right)=\int_M\frac{\partial}{\partial t}(dV_g)=-\frac{n}{2}\int_M(R_g-\overline{R}_g)dV_g=0$$
by (\ref{0}) and (\ref{3}),
which implies that
$\displaystyle\int_MdV_g=\int_MdV_{g_0}$ for all $t\geq 0$.
In particular, we have
\begin{equation}\label{20a}
\int_MdV_{g_\infty}=\int_MdV_{g_0}.
\end{equation}
On the other hand, note that  $R_{g_Y}=c^{\frac{4}{n-2}}R_{g_\infty}$
for some constant $c>0$. Indeed, one can take $c$ to be
$\displaystyle\left(R_{g_Y}/R_{g_\infty}\right)^{\frac{n-2}{4}}$.
This together with (\ref{Yamabe}) implies that the metric
$c^{\frac{4}{n-2}}g_Y$ has scalar curvature being equal to
$$R_{c^{\frac{4}{n-2}}g_Y}=c^{-\frac{4}{n-2}}R_{g_Y}=R_{g_\infty}.$$
Hence, we can conclude that
\begin{equation}\label{20b}
c^{\frac{4}{n-2}}g_Y=g_\infty
\end{equation}
using the result of Kazdan-Warner in \cite{Kazdan&Warner}  (see also  \cite{Lou}),
which says that
if $g_1$ and $g_2$ are two metrics conformal to $g_0$ such that $R_{g_1}=R_{g_2}<0$, then
$g_1=g_2$. Therefore, by (\ref{20a}) and (\ref{20b}), we have
$$\int_MdV_{g_0}=\int_MdV_{g_\infty}=\int_MdV_{c^{\frac{4}{n-2}}g_Y}=c^{\frac{2n}{n-2}}\int_MdV_{g_Y}=
c^{\frac{2n}{n-2}}\int_MdV_{g_0}$$
where the last equality follows from the assumption that $g_Y$ and $g_0$ have the same volume.
This implies that $c=1$, or equivalently,
\begin{equation}\label{20c}
g_Y=g_\infty.
\end{equation}

Note that by Proposition \ref{prop5} and Proposition \ref{prop6}, we have
\begin{equation}\label{20d}
\min_{M}R_{g_0}\leq \overline{R}_{g(t)}\leq\max_{M}R_{g_0} \mbox{ for all }t\geq 0.
\end{equation}
It follows from (\ref{20d}) and  Lemma \ref{thm1} that
\begin{equation}\label{18}
\begin{split}
&\left(\max_MR_{g_0}\right)\int_0^{t}\left(\max_MR_{g(s)}-\overline{R}_{g(s)}\right)ds\\
&\geq\left(\overline{R}_{g(t)}-\overline{R}_{g_0}\right)+\left(\max_MR_{g(t)}-\overline{R}_{g(t)}\right)-\left(\max_MR_{g_0}-\min_MR_{g_0}\right)\\
&\geq \left(\max_MR_{g(t)}-\overline{R}_{g(t)}\right)-2\left(\max_MR_{g_0}-\min_MR_{g_0}\right).
\end{split}
\end{equation}
Therefore, as $t\to\infty$,
by (\ref{18}) and Ye's result stated above that $g(t)\to g_\infty$ as $t\to\infty$, we get
\begin{equation}\label{19}
-2\left(1-\frac{\min_MR_{g_0}}{\max_MR_{g_0}}\right)\geq \int_0^{\infty}\left(\max_MR_{g(s)}-\overline{R}_{g(s)}\right)ds.
\end{equation}
Similarly, we obtain from (\ref{20d}) and Lemma \ref{thm2} that
\begin{equation*}
\begin{split}
&-\left(\max_MR_{g_0}\right)\int_0^{t}\left(\min_MR_{g(s)}-\overline{R}_{g(s)}\right)ds\\
&\geq
\left(\overline{R}_{g_0}-\overline{R}_{g(t)}\right)+
\left(\overline{R}_{g(t)}-\min_MR_{g(t)}\right)-\left(\max_MR_{g_0}-\min_MR_{g_0}\right)\\
&\geq \left(\overline{R}_{g(t)}-\min_MR_{g(t)}\right)-2\left(\max_MR_{g_0}-\min_MR_{g_0}\right).
\end{split}
\end{equation*}
Letting $t\to\infty$, we obtain
\begin{equation}\label{19a}
-2\left(1-\frac{\min_MR_{g_0}}{\max_MR_{g_0}}\right)\geq -\int_0^{\infty}\left(\min_MR_{g(s)}-\overline{R}_{g(s)}\right)ds.
\end{equation}
Integrating the  inequality in Proposition \ref{prop2}
and using (\ref{20c}), (\ref{19}) and (\ref{19a}), we get
\begin{equation*}
\begin{split}
\log\frac{\lambda_1(g_Y)}{\lambda_1(g_0)}&=\log\frac{\lambda_1(g_\infty)}{\lambda_1(g_0)}\\
&\geq-\frac{(n-2)}{2}\int_0^{\infty}\Big(\max_MR_{g(s)}-\overline{R}_{g(s)}\Big)ds
+\frac{n}{2}\int_0^{\infty}\Big(\min_MR_{g(s)}-\overline{R}_{g(s)}\Big)ds\\
&\geq 2(n-1)\left(1-\frac{\min_MR_{g_0}}{\max_MR_{g_0}}\right)
\end{split}
\end{equation*}
which gives the upper bound for $\lambda_1(g_0)$ in (\ref{20}).
We remark that the integration holds since
the Dini derivative is finite (see \cite{Hagood&Thomson} for example).
Similarly, integrating
the inequality in Proposition \ref{prop21} and using (\ref{20c}), (\ref{19}) and (\ref{19a}), we obtain
\begin{equation*}
\begin{split}
\log\frac{\lambda_1(g_Y)}{\lambda_1(g_0)}&=\log\frac{\lambda_1(g_\infty)}{\lambda_1(g_0)}\\
&\leq-\frac{(n-2)}{2}\int_0^{\infty}\Big(\min_MR_{g(s)}-\overline{R}_{g(s)}\Big)ds
+\frac{n}{2}\int_0^{\infty}\Big(\max_MR_{g(s)}-\overline{R}_{g(s)}\Big)ds\\
&\leq-2(n-1)\left(1-\frac{\min_MR_{g_0}}{\max_MR_{g_0}}\right)
\end{split}
\end{equation*}
which gives the lower bound for $\lambda_1(g_0)$ in (\ref{20}).
This proves the assertion.
\end{proof}

One can apply Theorem \ref{main} to obtain estimate of the first eigenvalue.
It was proved by Li and Yau (see Theorem 7 in \cite{Li&Yau}) that if $(M,g)$ is an $n$-dimensional
compact Riemannian manifold without boundary such that
its Ricci curvature satisfies
$Ric_g\geq(n-1)\kappa g$ where $\kappa<0$, then its first eigenvalue
satisfies
$$\lambda_1(g)\geq\frac{1}{(n-1)d^2}\exp\{-1-\sqrt{1+4(n-1)^2d^2|\kappa|}\}$$
where $d$ is the diameter of $(M,g)$.

\begin{theorem}\label{main1}
Suppose $M$ is an $n$-dimensional
compact manifold without boundary,
and $g_E$ is an Einstein metric on $M$ with
$Ric_{g_E}=(n-1)\kappa\,g_E$ where $\kappa<0$.
If $g_0$ is a Riemannian metric  conformal to $g_E$
which has negative scalar curvature and same volume as $g_E$, then
the first eigenvalue of $(M,g_0)$ satisfies
$$\lambda_1(g_0)\geq \frac{1}{(n-1)e^\frac{nc}{n-1}d(M,g_0)^2}\exp\Big\{-1-\sqrt{1+4(n-1)^2e^\frac{c}{n-1}d(M,g_0)^2|\kappa|}\Big\}$$
where $c=2(n-1)\displaystyle\left(\frac{\min_MR_{g_0}}{\max_MR_{g_0}}-1\right).$
\end{theorem}
\begin{proof}
As in the proof of Theorem \ref{main}, we can show that
$g\to g_\infty=g_E$ as $t\to\infty$
under the Yamabe flow (\ref{1}).
We claim that
\begin{equation}\label{22a}
e^{-\frac{c}{2(n-1)}}d(M,g_E)\leq d(M,g_0)\leq
e^{\frac{c}{2(n-1)}}d(M,g_E),
\end{equation}
where $c=\displaystyle2(n-1)\left(\frac{\min_MR_{g_0}}{\max_MR_{g_0}}-1\right)$ as in Theorem \ref{main},
$d(M,g_0)$ and $d(M,g_E)$ are the diameter of $M$ with respect to the initial metric $g_0$ and
the Einstein metric $g_E$ respectively. To see this,
we let $\gamma: [s_0,s_1]\to M$ be a differentiable curve joining $x$ and $y$ in $M$.
Consider the solution $g$ of the Yamabe flow (\ref{1}) with $g_0$ as the initial metric.
Then the length of
$\gamma$ with respect to the metric $g$ is given by
$$L_g(\gamma)=\int_{s_0}^{s_1}\sqrt{g(\frac{d\gamma}{ds},\frac{d\gamma}{ds})}\,ds.$$
Differentiate it with respect $t$, we obtain
\begin{equation}\label{21}
\begin{split}
\frac{dL_g(\gamma)}{dt}&=\frac{d}{dt}\left(\int_{s_0}^{s_1}u^{\frac{2}{n-2}}\sqrt{g_0(\frac{d\gamma}{ds},\frac{d\gamma}{ds})}\,ds\right)\\
&=\int_{s_0}^{s_1}\frac{2}{n-2}u^{\frac{2}{n-2}-1}\frac{\partial u}{\partial t}\sqrt{g_0(\frac{d\gamma}{ds},\frac{d\gamma}{ds})}\,ds\\
&=\frac{1}{2}\int_{s_0}^{s_1}(\overline{R}_g-R_g)\sqrt{g(\frac{d\gamma}{ds},\frac{d\gamma}{ds})}\,ds\\
&\leq \frac{1}{2}\left(\overline{R}_g-\min_MR_g\right)L_g(\gamma)
\end{split}
\end{equation}
where we have used (\ref{2}) and the fact that $g=u^{\frac{4}{n-2}}g_0$. Similarly, we can get
\begin{equation}\label{22}
\frac{dL_g(\gamma)}{dt}\geq\frac{1}{2}\left(\overline{R}_g-\max_MR_g\right)L_g(\gamma).
\end{equation}
Integrating (\ref{21}) and (\ref{22}) from $0$ to $\infty$, we obtain
$$\frac{1}{2}\int_0^{\infty}\left(\overline{R}_{g(t)}-\max_MR_{g(t)}\right)dt\leq \log\frac{L_{g_E}(\gamma)}{L_{g_0}(\gamma)}\leq \frac{1}{2}\int_0^{\infty}\left(\overline{R}_{g(t)}-\min_MR_{g(t)}\right)dt.$$
Combining this with (\ref{19}) and (\ref{19a}), we obtain
$$\left(1-\frac{\min_MR_{g_0}}{\max_MR_{g_0}}\right)\leq \log\frac{L_{g_E}(\gamma)}{L_{g_0}(\gamma)}\leq -\left(1-\frac{\min_MR_{g_0}}{\max_MR_{g_0}}\right).$$
This implies
$$\left(1-\frac{\min_MR_{g_0}}{\max_MR_{g_0}}\right)\leq \log\frac{d(M,g_E)}{d(M,g_0)}\leq -\left(1-\frac{\min_MR_{g_0}}{\max_MR_{g_0}}\right)$$
which proves the claim (\ref{22a}).

By the assumption and the result of Li and Yau mentioned above, we have
$$\lambda_1(g_E)\geq\frac{1}{(n-1)d(M,g_E)^2}\exp\Big\{-1-\sqrt{1+4(n-1)^2d(M,g_\infty)^2|\kappa|}\Big\}.$$
Combining this with (\ref{22a}) and Theorem \ref{main}, we obtain
$$e^{c}\lambda_1(g_0)\geq\frac{1}{(n-1)e^\frac{c}{n-1}d(M,g_0)^2}\exp\Big\{-1-\sqrt{1+4(n-1)^2e^\frac{c}{n-1}d(M,g_0)^2|\kappa|}\Big\}.$$
This proves the assertion.
\end{proof}

\section{The unnormalized Yamabe flow on manifolds without boundary}

Now we consider the unnormalized Yamabe flow on an $n$-dimensional compact Riemannian manifold $(M,g_0)$ without boundary,
which is defined as
\begin{equation}\label{C1}
\frac{\partial}{\partial t}g=-R_gg\mbox{ for }t\geq 0,\hspace{2mm}g|_{t=0}=g_0.
\end{equation}
If we write $g=u^{\frac{4}{n-2}}g_0$ for some $0<u\in C^\infty(M)$, then $u$ satisfies the following evolution equation:
\begin{equation}\label{C2}
\frac{\partial u}{\partial t}=-\frac{n-2}{4}R_gu\mbox{ for }t\geq 0,\hspace{2mm}u|_{t=0}=1.
\end{equation}
Hence, the volume form $dV_g$ and the scalar curvature $R_g$ of $g$ satisfy (see \cite{Chow})
\begin{eqnarray}\label{C3}
&&\frac{\partial}{\partial t}dV_g=-\frac{n}{2}R_gdV_g,\\
\label{C4}
&&\frac{\partial}{\partial t}R_g=(n-1)\Delta_gR_g+R_g^2.
\end{eqnarray}

Let $\lambda_1$ be the first eigenvalue of $-\Delta_g+aR_g$ where $a$ is a constant, i.e.
\begin{equation}\label{C5}
-\Delta_g f+aR_gf=\lambda_1 f
\end{equation}
for some function $f$. By the eigenvalue perturbation theory
(see \cite{Reed&Simon} and also \cite{Cao,Kato,Kleiner&Lott}), we may assume that
there is a family of the first eigenvalue and the corresponding eigenfunction
which is $C^1$ in $t$.
By rescaling, we may assume that the eigenfunction $f$ satisfies
\begin{equation}\label{C6}
\int_Mf^2dV_g=1.
\end{equation}

\begin{prop}\label{propC1}
Along the unnormalized Yamabe flow (\ref{C1}), we have
\begin{equation*}
\begin{split}
\frac{d\lambda_1}{dt}&=\left(2(n-1)a-\frac{n-2}{2}\right)\int_MR_g(|\nabla_{g} f|^2_{g}+aR_gf^2)dV_{g}\\
&\hspace{4mm}
-\left(2(n-1)a-\frac{n}{2}\right)\lambda_1\int_MR_gf^2dV_g.
\end{split}
\end{equation*}
\end{prop}
\begin{proof}
Differentiate (\ref{C6}) with respect to $t$, we have
\begin{equation}\label{C7}
\int_Mf\frac{\partial f}{\partial t}dV_g=\frac{n}{4}\int_MR_gf^2dV_g
\end{equation}
by (\ref{C3}).
Multiply (\ref{C5}) by $f$ and integrate it over $M$, we obtain
\begin{equation}\label{C8}
\lambda_1=\int_M(|\nabla_g f|^2_g+aR_gf^2)dV_g
\end{equation}
by (\ref{C6}) and integration by parts. Since $g=u^{\frac{4}{n-2}}g_0$,
\eqref{9} holds.
Combining (\ref{9}) and (\ref{C8}), we have
$$\lambda_1=\int_Mu^2|\nabla_{g_0} f|^2_{g_0}dV_{g_0}+a\int_MR_gf^2dV_g.$$
Differentiate it with respect to $t$, we obtain
\begin{equation*}
\begin{split}
\frac{d\lambda_1}{dt}&=\int_M2u\frac{\partial u}{\partial t}|\nabla_{g_0} f|^2_{g_0}dV_{g_0}
+2\int_Mu^2\langle\nabla_{g_0} f,\nabla_{g_0}(\frac{\partial f}{\partial t})\rangle_{g_0}dV_{g_0}\\
&\hspace{4mm}+a\int_Mf^2\frac{\partial R_g}{\partial t}dV_g+2a\int_MR_gf\frac{\partial f}{\partial t}dV_g
+a\int_MR_gf^2\frac{\partial}{\partial t}(dV_g)\\
&=-\frac{n-2}{2}\int_MR_g|\nabla_{g} f|^2_{g}dV_{g}
+2\int_M\langle\nabla_{g} f,\nabla_{g}(\frac{\partial f}{\partial t})\rangle_{g}dV_{g}\\
&\hspace{4mm}+a\int_Mf^2\big((n-1)\Delta_gR_g+R_g^2\big)dV_g+2a\int_MR_gf\frac{\partial f}{\partial t}dV_g
-\frac{n}{2}a\int_MR_g^2f^2dV_g\\
&=-\frac{n-2}{2}\int_MR_g|\nabla_{g} f|^2_{g}dV_{g}
+2\int_M\frac{\partial f}{\partial t}\big(-\Delta_gf+aR_gf\big)dV_g\\
&\hspace{4mm}+(n-1)a\int_MR_g\Delta_g(f^2)dV_g
-\frac{n-2}{2}a\int_MR_g^2f^2dV_g\\
&=\left(2(n-1)a-\frac{n-2}{2}\right)\int_MR_g|\nabla_{g} f|^2_{g}dV_{g}
+2\lambda_1\int_Mf\frac{\partial f}{\partial t}dV_g\\
&\hspace{4mm}-2(n-1)a\int_MR_gf\big(-\Delta_gf+aR_gf\big)dV_g\\
&\hspace{4mm}
+\left(2(n-1)a^2-\frac{n-2}{2}a\right)\int_MR_g^2f^2dV_g\\
&=\left(2(n-1)a-\frac{n-2}{2}\right)\int_MR_g(|\nabla_{g} f|^2_{g}+aR_gf^2)dV_{g}\\
&\hspace{4mm}
-\left(2(n-1)a-\frac{n}{2}\right)\lambda_1\int_MR_gf^2dV_g
\end{split}
\end{equation*}
where the second equality follows from  (\ref{9}) and (\ref{C2})-(\ref{C4}),
the third equality follows from integration by parts, and the last two equalities follow from (\ref{C5}) and (\ref{C7}).
This proves the assertion.
\end{proof}

\begin{prop}\label{propC2}
If $0\leq a\leq\displaystyle\frac{n-2}{4(n-1)}$ and
$\displaystyle\min_MR_g\geq\frac{n-2}{n}\max_MR_g\geq 0$, then
$\displaystyle\frac{d\lambda_1}{dt}\geq 0$, and equality holds if and only if $R_g$ is constant.
\end{prop}
\begin{proof}
If $a\leq\displaystyle\frac{n-2}{4(n-1)}$, then
$\displaystyle\frac{n}{2}-2(n-1)a\geq\frac{n-2}{2}-2(n-1)a\geq 0.$
Combining this with Proposition \ref{propC1}, we obtain
\begin{equation*}
\begin{split}
\frac{d\lambda_1}{dt}
&\geq-\left(\frac{n-2}{2}-2(n-1)a\right)\Big(\max_MR_g\Big)\int_M(|\nabla_{g} f|^2_{g}+aR_gf^2)dV_{g}\\
&\hspace{4mm}
+\left(\frac{n}{2}-2(n-1)a\right)\Big(\min_MR_g\Big)\lambda_1\int_Mf^2dV_g\\
&=\lambda_1\left(\frac{n}{2}\min_MR_g-\frac{n-2}{2}\max_MR_g\right)
\end{split}
\end{equation*}
where we have used (\ref{C6}) and (\ref{C8}) in the last equality.  From this, the assertion follows.
\end{proof}

\begin{prop}\label{propC4}
If $\displaystyle \frac{n-2}{4(n-1)}\leq a\leq \frac{n}{4(n-1)}$ and
$\displaystyle\min_MR_g\geq 0$, then
$\displaystyle\frac{d\lambda_1}{dt}\geq 0$
and equality holds if and only if $R_g\equiv 0$.
\end{prop}
\begin{proof}
If $\displaystyle \frac{n-2}{4(n-1)}\leq a\leq \frac{n}{4(n-1)}$, then
$\displaystyle2(n-1)a-\frac{n-2}{2}\geq 0$
and
$\displaystyle\frac{n}{2}-2(n-1)a\geq 0$. Combining this with Proposition \ref{propC1}, we get
\begin{equation*}
\begin{split}
\frac{d\lambda_1}{dt}&\geq\left(2(n-1)a-\frac{n-2}{2}\right)\Big(\min_MR_g\Big)\int_M(|\nabla_{g} f|^2_{g}+aR_gf^2)dV_{g}\\
&\hspace{4mm}
+\left(\frac{n}{2}-2(n-1)a\right)\Big(\min_MR_g\Big)\lambda_1\int_Mf^2dV_g\\
&=\lambda_1\Big(\min_MR_g\Big)
\end{split}
\end{equation*}
where we have used (\ref{C6}) and (\ref{C8}).
From this, the assertion follows.
\end{proof}

\begin{prop}\label{propC3}
If $a\geq\displaystyle\frac{n}{4(n-1)}$ and
$\displaystyle\min_MR_g\geq 0$, then
$\displaystyle\frac{d\lambda_1}{dt}\geq 0$
and equality holds if and only if $R_g\equiv 0$.
\end{prop}
\begin{proof}
If $a\geq\displaystyle\frac{n}{4(n-1)}$, then
$\displaystyle2(n-1)a-\frac{n-2}{2}\geq2(n-1)a-\frac{n}{2}\geq 0.$
Combining this with Proposition \ref{propC1}, we obtain
\begin{equation*}
\begin{split}
\frac{d\lambda_1}{dt}
&\geq\left(2(n-1)a-\frac{n-2}{2}\right)\Big(\min_MR_g\Big)\int_M(|\nabla_{g} f|^2_{g}+aR_gf^2)dV_{g}\\
&\hspace{4mm}
-\left(2(n-1)a-\frac{n}{2}\right)\Big(\max_MR_g\Big)\lambda_1\int_Mf^2dV_g\\
&=\lambda_1\left(\frac{n}{2}\max_MR_g-\frac{n-2}{2}\min_MR_g\right)
\end{split}
\end{equation*}
where the last equality follows from (\ref{C6}) and (\ref{C8}). From this, the assertion follows.
\end{proof}

\begin{proof}[Proof of Theorem \ref{Mmain}]
This follows from Proposition \ref{propC2}-\ref{propC3}.
\end{proof}

\section{The Yamabe flow on manifolds with boundary}

Throughout this section, we assume that $(M,g_0)$ is a compact Riemannian manifold of dimension $n\geq 3$ with smooth boundary $\partial M$.
Up to a conformal change, we may assume that the mean curvature of $g_0$
on $\partial M$ vanishes. See Lemma 2.1 in \cite{Brendle1} for the proof.
We consider the Yamabe flow, which is defined as
\begin{equation}\label{B1}
\frac{\partial}{\partial t}g=-(R_g-\overline{R}_g)g\mbox{ in }M\mbox{ and }H_g=0 \hbox{ on }\partial M
\mbox{ for }t\geq 0,\hspace{2mm}g|_{t=0}=g_0.
\end{equation}
Here  $H_g$ is the mean curvature of $g$ with respect to the outward unit normal $\nu_g$.
If we write $g=u^{\frac{4}{n-2}}g_0$, then $u$ satisfies
\begin{equation}\label{B2}
\frac{\partial u}{\partial t}=-\frac{n-2}{4}(R_g-\overline{R}_g)u\mbox{ in }M\mbox{ for }t\geq 0.
\end{equation}
Under the Yamabe flow (\ref{B1}), the volume form $dV_g$ and
the scalar curvature $R_g$ of $g$ satisfy (see \cite{Brendle1})
\begin{eqnarray}\label{B3}
&&\frac{\partial}{\partial t}(dV_g)=-\frac{n}{2}(R_g-\overline{R}_g)dV_g,\\
&&\label{B4}
\frac{\partial}{\partial t}R_g=(n-1)\Delta_gR_g-R_g(\overline{R}_g-R_g).
\end{eqnarray}

We have the following:

\begin{lem}\label{BAUX}
Let $g=g(t)$ be the solution of the
the Yamabe flow $(\ref{B1})$
and $\lambda_1(t)$ be the corresponding first nonzero Dirichlet eigenvalue of the Laplacian.
Then for any $t_2\geq t_1$,
there exists a $C^\infty$ function $f$ on $M\times [t_1,t_2]$
satisfying
\begin{equation}\label{Bcondtion}
f=0\mbox{ on }\partial M\hspace{2mm}\mbox{
and }\hspace{2mm}
\int_{M}f^2dV_{g}=1\mbox{ for all }t,
\end{equation}
and
$$\lambda_1(t_2)\geq \lambda_1(t_1)
-\frac{n-2}{2}\int_{t_1}^{t_2}\int_M(R_g-\overline{R}_g)|\nabla_{g} f|^2_{g}dV_{g}dt-2\int_{t_1}^{t_2}\int_{M}\frac{\partial f}{\partial t}\Delta_{g} fdV_{g}dt$$
such that at time $t_2$, $f(t_2)$ is the corresponding eigenfunction of
$\lambda_1(t_2)$.
\end{lem}
\begin{proof}
As in the proof of Lemma \ref{AUX},
we choose $f_2 = f(t_2)$ to be the eigenfunction for the
first nonzero Dirichlet eigenvalue
$\lambda_1(t_2)$ of $g(t_2)$.
Then $f_2$ satisfies the Dirichlet boundary condition, i.e. $f_2=0$ on $\partial M$.
Thus if we define
$$h(t)=\left(\frac{u(t_2)}{u(t)}\right)^{\frac{2n}{n-2}}f_2$$
where $u(t)$ is the solution of
(\ref{B2}),
then the
normalized function
$$f(t)=\frac{h(t)}{(\int_{M}h(t)^2dV_{g(t)})^{\frac{1}{2}}}$$
satisfies (\ref{Bcondtion}).
Now we can follow the same proof of Lemma \ref{AUX}
to finish the proof, except we have to
use the fact that  $f=0$ on $\partial M$
when we do the integration by parts in the last equality in
(\ref{3.13A}).
This proves the assertion.
\end{proof}

Once Lemma \ref{BAUX} is proved, we can follow the same proof of Proposition \ref{prop2}
to prove the following:

\begin{prop}\label{Bprop2}
The first nonzero Dirichlet eigenvalue $\lambda_1$ of $g$ along the Yamabe flow
$(\ref{B1})$ satisfies
$$\frac{d}{dt}\log\lambda_1\geq-\frac{n-2}{2}\Big(\max_MR_g-\overline{R}_g\Big)
+\frac{n}{2}\Big(\min_MR_g-\overline{R}_g\Big).$$
Here the derivative on the left hand side
is in the
sense of the $\liminf$ of backward difference quotients.
\end{prop}

Similar to the case when the manifold has no boundary,
we can prove the corresponding version of Lemma \ref{BAUX}.
Then we can follow the same proof of Proposition \ref{prop21} to prove the following:
\begin{prop}\label{Bprop21}
The first nonzero Dirichlet eigenvalue $\lambda_1$ of $g$ along the Yamabe flow
$(\ref{1})$ satisfies
$$\frac{d}{dt}\log\lambda_1\leq-\frac{n-2}{2}\Big(\min_MR_g-\overline{R}_g\Big)
+\frac{n}{2}\Big(\max_MR_g-\overline{R}_g\Big).$$
Here the derivative on the left hand side
is in the
sense of the $\limsup$ of forward difference quotients.
\end{prop}

On the other hand,
one can apply the maximum principle to prove
Proposition \ref{prop4}-\ref{prop6} and
Lemma \ref{thm1}-\ref{thm2} for the
Yamabe flow (\ref{B1}) when the manifold has boundary.

Therefore, if $(M,g_0)$ is a compact Riemannian manifold of dimension $n\geq 3$ with smooth boundary such that
$\displaystyle\max_MR_{g_0}< 0$, and  $g_Y$ is the Yamabe metric in the conformal class of $g_0$, i.e. $g_Y$ is the
Riemannian metric conformal to $g_0$ such that its scalar curvature is constant in $M$
and its mean curvature is zero on $\partial M$, then we have the following:

\begin{theorem}\label{mainB}
Suppose $(M,g_0)$ is a compact Riemannian manifold of dimension $n\geq 3$ with smooth boundary $\partial M$
which has negative scalar curvature in $M$ and vanishing mean curvature on $\partial M$,
and $g_Y$ is the Yamabe metric conformal to $g_0$
which has same volume as $g_0$.
Then the first nonzero Dirichlet eigenvalue of $g_0$ and $g_Y$ satisfy
$$
e^{-c}\lambda_1(g_Y)\leq\lambda_1(g_0)\leq
e^{c}\lambda_1(g_Y)
$$
where $c$ is the constant in Theorem \ref{main}.
\end{theorem}
\begin{proof}
We only sketch the proof since it is essentially the same as the proof of Theorem \ref{main}.
Brendle has proved in \cite{Brendle1} that $g\to g_\infty$ as $t\to\infty$
under the Yamabe flow (\ref{B1}) such that $g_\infty$ has constant scalar curvature in $M$
and vanishing mean curvature on $\partial M$ (c.f. Theorem 1.1 in \cite{Brendle1}).
As in the proof of Theorem \ref{main}, we can prove that $g_\infty=g_Y$
by using the result of Escobar (see Corollary in \cite{Escobar}), which says that if $g_1$ and $g_2$ are two metrics conformal to $g_0$ such that
$R_{g_1}=R_{g_2}<0$ in $M$ and $H_{g_1}=H_{g_2}=0$ on $\partial M$, then
$g_1=g_2$.
On the other hand, we can follow the same arguments as
in the case without boundary to get (\ref{19}) and (\ref{19a}).
The remaining arguments are the same as the proof of Theorem \ref{main}.
This proves the assertion.
\end{proof}

One can apply Theorem \ref{mainB} to obtain estimate of the  first nonzero Dirichlet eigenvalue.
In \cite{Ling},
Ling proved the following:
Let $(M,g_0)$ be an $n$-dimensional compact Riemannian manifold
with boundary. Suppose that the boundary $\partial M$ has nonnegative mean curvature with respect to the outward normal and that the
Ricci curvature of $M$ has lower bound
$Ric(M)\geq(n-1)\kappa$. Then the first nonzero Dirichlet eigenvalue $\lambda_1$ of the Laplacian of $M$ satisfies
$$\lambda_1\geq\frac{1}{2}(n-1)\kappa+\frac{\pi^2}{d^2},$$
where $d$ is the diameter of the largest interior ball in $M$.

\begin{theorem}
Suppose $M$ is an $n$-dimensional
compact manifold with smooth boundary $\partial M$,
and $g_E$ is an Einstein metric on $M$ with
$Ric_{g_E}=(n-1)\kappa\,g_E$ where $\kappa<0$
and vanishing mean curvature on $\partial M$.
If $g_0$ is a Riemannian metric  conformal to $g_E$
which has negative scalar curvature in $M$, vanishing mean curvature on $\partial M$, and same volume as $g_E$, then
the first nonzero Dirichlet eigenvalue of $(M,g_0)$
satisfies
$$\lambda_1(g_0)\geq \frac{1}{2e^{c}}(n-1)\kappa+\frac{\pi^2}{e^{\frac{nc}{n-1}}d(M,g_0)^2}$$
where $c=2(n-1)\displaystyle\left(\frac{\min_MR_{g_0}}{\max_MR_{g_0}}-1\right)$
and $d(M,g_0)$ is the diameter of the largest interior ball in $M$.
\end{theorem}
\begin{proof}
As in the proof of Theorem \ref{main1}, we can prove that
\begin{equation}\label{B12}
e^{-\frac{c}{2(n-1)}}d(M,g_E)\leq d(M,g_0)\leq
e^{\frac{c}{2(n-1)}}d(M,g_E),
\end{equation}
where $c=\displaystyle2(n-1)\left(1-\frac{\min_MR_{g_0}}{\max_MR_{g_0}}\right)$, and
$d(M,g_0)$ and $d(M,g_E)$ are respectively the  diameter of the largest interior ball in $M$ with respect to the initial metric $g_0$ and
the Einstein metric $g_\infty$.
On the other hand, by the assumptions and the above result of Ling, we have
$$\lambda_1(g_E)\geq\frac{1}{2}(n-1)\kappa+\frac{\pi^2}{d(M,g_E)^2}.$$
Combining this with (\ref{B12}) and Theorem \ref{mainB}, we obtain
$$e^{c}\lambda_1(g_0)\geq \frac{1}{2}(n-1)\kappa+\frac{\pi^2}{e^{\frac{c}{n-1}}d(M,g_0)^2}.$$
This proves the assertion.
\end{proof}

\section{The unnormalized Yamabe flow on manifolds with boundary}

In this section,
we study the unnormalized Yamabe flow
on a compact Riemannian manifold $(M,g_0)$ with smooth boundary $\partial M$, which is defined as
\begin{equation}\label{MB1}
\frac{\partial}{\partial t}g=-R_gg\mbox{ in }M\mbox{ and }H_g=0 \hbox{ on }\partial M
\mbox{ for }t\geq 0,\hspace{2mm}g|_{t=0}=g_0.
\end{equation}
If we write $g=u^{\frac{4}{n-2}}g_0$, then
\begin{equation}\label{MB2}
\frac{\partial u}{\partial t}=-\frac{n-2}{4}R_gu\mbox{ in }M\mbox{ for }t\geq 0.
\end{equation}
Note that the volume form $dV_g$ and the scalar curvature $R_g$ of $g$ satisfy
\begin{eqnarray}\label{MB3}
&&\frac{\partial}{\partial t}(dV_g)=-\frac{n}{2}R_gdV_g,\\
&&\label{MB4}
\frac{\partial}{\partial t}R_g=(n-1)\Delta_gR_g+R_g^2.
\end{eqnarray}
along the unnormalized Yamabe flow (\ref{MB1}).

Let $\lambda_1$ be the first eigenvalue of $-\Delta_g+aR_g$ with Dirichlet boundary condition, i.e.
\begin{equation}\label{MB5}
-\Delta_g f+aR_gf=\lambda_1 f\mbox{ in }M\hspace{2mm}\mbox{ and }f=0\mbox{ on }\partial M
\end{equation}
for some function $f$.
Again we assume that
there is a family of the first eigenvalue and the corresponding eigenfunction which is $C^1$ in $t$.
By rescaling, we may assume that the eigenfunction $f$ satisfies
\begin{equation}\label{MB6}
\int_Mf^2dV_g=1.
\end{equation}

\begin{prop}\label{MpropB1}
Along the unnormalized Yamabe flow \eqref{MB1}, we have
\begin{equation*}
\begin{split}
\frac{d\lambda_1}{dt}&=\left(2(n-1)a-\frac{n-2}{2}\right)\int_MR_g(|\nabla_{g} f|^2_{g}+aR_gf^2)dV_{g}\\
&\hspace{4mm}
-\left(2(n-1)a-\frac{n}{2}\right)\lambda_1\int_MR_gf^2dV_g.
\end{split}
\end{equation*}
\end{prop}
\begin{proof}
The proof is almost identical to the proof of Proposition \ref{propC1} except
we have to take care of the boundary term when we integrate by parts. More precisely,
we multiply the first equation in (\ref{MB5}) by $f$ and integrate it over $M$, we obtain
\begin{equation}\label{MB8}
\begin{split}
\lambda_1=\int_M(-f\Delta_g f+aR_gf^2)dV_g&=\int_M(|\nabla_g f|^2_g+aR_gf^2)dV_g-\int_{\partial M}f\frac{\partial f}{\partial\nu_g}dV_g\\
&=\int_M(|\nabla_g f|^2_g+aR_gf^2)dV_g
\end{split}
\end{equation}
by (\ref{MB5}).
On the other hand, we have
\begin{equation}\label{MB9}
\int_M\langle\nabla_{g} f,\nabla_{g}(\frac{\partial f}{\partial t})\rangle_{g}dV_{g}
=-\int_M\frac{\partial f}{\partial t}\Delta_{g} fdV_{g}+\int_{\partial M}\frac{\partial f}{\partial t}\frac{\partial f}{\partial \nu_g}dV_g
=-\int_M\frac{\partial f}{\partial t}\Delta_{g} fdV_{g}
\end{equation}
since $\displaystyle\frac{\partial f}{\partial t}=0$ on $\partial M$ by (\ref{MB5}),  and
\begin{equation}\label{MB10}
\begin{split}
\int_Mf^2\Delta_gR_gdV_{g}-\int_MR_g\Delta_g(f^2)dV_{g}
&=\int_{\partial M}f^2\frac{\partial R_g}{\partial\nu_g}dV_{g}-\int_{\partial M}R_g\frac{\partial}{\partial \nu_g}(f^2)dV_{g}\\
&=\int_{\partial M}f^2\frac{\partial R_g}{\partial\nu_g}dV_{g}-2\int_{\partial M}R_gf\frac{\partial f}{\partial \nu_g}dV_{g}=0
\end{split}
\end{equation}
by (\ref{MB5}). Except these, all the other steps are the same as the proof of Proposition \ref{propC1}.
This proves the assertion.
\end{proof}

Now, by the same proof of Proposition \ref{propC2}-\ref{propC3}, we have the following:

\begin{theorem}\label{Mmain3}
Along the unnormalized Yamabe flow \eqref{MB1},
the first eigenvalue of $-\Delta_g+aR_g$
with Dirichlet boundary condition is nondecreasing\\
(i)
if $0\leq a<\displaystyle\frac{n-2}{4(n-1)}$ and
$\displaystyle\min_MR_g\geq\frac{n-2}{n}\max_MR_g\geq 0$;\\
(ii) if $a\geq\displaystyle\frac{n-2}{4(n-1)}$ and
$\displaystyle\min_MR_g\geq 0$.
\end{theorem}

\subsection{Neumann boundary condition}

Let $\mu_1$ be the first eigenvalue of
$-\Delta_g+aR_g$ with Neumann boundary condition, i.e.
\begin{equation}\label{MB11}
-\Delta_g f+aR_gf=\mu_1 f\mbox{ in }M\hspace{2mm}\mbox{ and }\frac{\partial f}{\partial\nu_g}=0\mbox{ on }\partial M.
\end{equation}
Again we may assume that
there is a family of the first eigenvalue and the corresponding eigenfunction which is $C^1$ in $t$.
By rescaling, we may further assume that the eigenfunction $f$ satisfies
\begin{equation*}
\int_Mf^2dV_g=1.
\end{equation*}
Note that
\eqref{MB8}-\eqref{MB10} are still true thanks to \eqref{MB11}.
Thus, following the proof of Proposition \ref{MpropB1}, we get
\begin{equation*}
\begin{split}
\frac{d\mu_1}{dt}&=\left(2(n-1)a-\frac{n-2}{2}\right)\int_MR_g(|\nabla_{g} f|^2_{g}+aR_gf^2)dV_{g}\\
&\hspace{4mm}
-\left(2(n-1)a-\frac{n}{2}\right)\mu_1\int_MR_gf^2dV_g.
\end{split}
\end{equation*}
along the unnormalized Yamabe flow \eqref{MB1}. As for the Dirichlet boundary condition, we have
the following:

\begin{theorem}\label{Mmain4}
Along the unnormalized Yamabe flow \eqref{MB1},
the first eigenvalue of $-\Delta_g+aR_g$
with Neumann boundary condition is nondecreasing\\
(i)
if $0\leq a<\displaystyle\frac{n-2}{4(n-1)}$ and
$\displaystyle\min_MR_g\geq\frac{n-2}{n}\max_MR_g\geq 0$;\\
(ii) if $a\geq\displaystyle\frac{n-2}{4(n-1)}$ and
$\displaystyle\min_MR_g\geq 0$.
\end{theorem}

\section{The CR Yamabe flow}

Throughout this section, we suppose that  $(M,\theta_0)$ is a compact strictly pseudoconvex CR manifold of real dimension $2n+1$.
We consider the CR Yamabe flow, which is defined as
\begin{equation}\label{A1}
\frac{\partial}{\partial t}\theta=-(R_\theta-\overline{R}_\theta)\theta\mbox{ for }t\geq 0,\hspace{2mm}\theta|_{t=0}=\theta_0.
\end{equation}
Here $R_\theta$ is the Webster scalar curvature of the contact form $\theta$, and $\overline{R}_\theta$ is the average of the Webster scalar curvature
given by
\begin{equation}\label{A0}
\overline{R}_\theta=\frac{\int_MR_\theta dV_\theta}{\int_MdV_\theta},
\end{equation}
where $dV_\theta=\theta\wedge (d\theta)^n$ is the volume form of $\theta$. Since the CR Yamabe flow preserves the conformal structure, we can write $\theta=u^{\frac{2}{n}}\theta_0$
for some positive function $u$, and $u$ satisfies the following evolution equation:
\begin{equation}\label{A2}
\frac{\partial u}{\partial t}=-\frac{n}{2}(R_\theta-\overline{R}_\theta)u\mbox{ for }t\geq 0,\hspace{2mm}u|_{t=0}=1.
\end{equation}
Hence, the volume form $dV_\theta$ of $\theta$ satisfies
\begin{equation}\label{A3}
\frac{\partial}{\partial t}(dV_\theta)=\frac{\partial}{\partial t}(u^{\frac{2n+2}{n}}dV_{\theta_0})=
\frac{2n+2}{n}u^{\frac{2n+2}{n}-1}\frac{\partial u}{\partial t}dV_{\theta_0}=-(n+1)(R_\theta-\overline{R}_\theta)dV_\theta.
\end{equation}
On the other hand, the Webster scalar curvature $R_\theta$ of $\theta$ satisfies the following evolution equation:
(see Proposition 3.2 in \cite{Ho2} or Lemma 2.4 in \cite{Ho1})
\begin{equation}\label{A4}
\frac{\partial}{\partial t}R_\theta=(n+1)\Delta_\theta R_\theta-R_\theta(\overline{R}_\theta-R_\theta).
\end{equation}
Here $\Delta_\theta$ is the sub-Laplacian of the contact form $\theta$.

We have the following lemma, which is again inspired by the Proposition 3.1 in \cite{Wu&Wang&Zheng}.

\begin{lem}\label{AAUX}
Let $\theta=\theta(t)$ be the solution of the
the CR Yamabe flow $(\ref{A1})$
and $\lambda_1(t)$ be the corresponding first nonzero eigenvalue of the sub-Laplacian.
Then for any $t_2\geq t_1$,
there exists a $C^\infty$ function $f$ on $M\times [t_1,t_2]$
satisfying
\begin{equation}\label{Acondition}
\int_{M}f^2dV_{\theta}=1\hspace{2mm}\mbox{
and }\hspace{2mm}\displaystyle\int_{M}fdV_{\theta}=0\mbox{ for all }t,
\end{equation}
and
$$\lambda_1(t_2)\geq \lambda_1(t_1)
-n\int_{t_1}^{t_2}\int_M(R_\theta-\overline{R}_\theta)|\nabla_{\theta} f|^2_{\theta}dV_{\theta}dt-2\int_{t_1}^{t_2}\int_{M}\frac{\partial f}{\partial t}\Delta_{\theta} fdV_{\theta}dt$$
such that at time $t_2$, $f(t_2)$ is the corresponding eigenfunction of
$\lambda_1(t_2)$.
\end{lem}
\begin{proof}
At time $t_2$, we  let $f_2 = f(t_2)$ be the eigenfunction for the
first nonzero eigenvalue
$\lambda_1(t_2)$ of $\theta(t_2)$.
We define the following smooth function on $M$:
$$h(t)=\left(\frac{u(t_2)}{u(t)}\right)^{2+\frac{2}{n}}f_2$$
where $u(t)$ is the solution of
(\ref{A2}).
We normalize this smooth function on $M$ by
$$f(t)=\frac{h(t)}{(\int_{M}h(t)^2dV_{\theta(t)})^{\frac{1}{2}}}.$$
Then we can easily check that $f(t)$ satisfies (\ref{Acondition}).

Set
$$
G(\theta(t),f(t)):=\int_M|\nabla_{\theta(t)}f(t)|^2_{\theta(t)}dV_{\theta(t)}.
$$
Note that $G(\theta(t),f(t))$ is a smooth function in $t$.
Since $\theta=u^{\frac{2}{n}}\theta_0$, we have
\begin{equation}\label{A9}
\langle\nabla_\theta f_1,\nabla_\theta f_2\rangle_\theta=u^{-\frac{2}{n}}\langle\nabla_{\theta_0} f_1,\nabla_{\theta_0} f_2\rangle_{\theta_0}
\end{equation}
for any functions $f_1, f_2$ in $M$, which implies that
\begin{equation*}
G(g(t),f(t))=\int_Mu^2|\nabla_{\theta_0} f|^2_{\theta_0}dV_{\theta_0}.
\end{equation*}
Differentiating it with respect to $t$, we get
\begin{equation}\label{6.8}
\begin{split}
\mathcal{G}(\theta(t),f(t))&:=\frac{d}{dt}G(\theta(t),f(t))\\
&=\int_M2u\frac{\partial u}{\partial t}|\nabla_{\theta_0} f|^2_{\theta_0}dV_{\theta_0}
+2\int_Mu^2\langle\nabla_{\theta_0} f,\nabla_{\theta_0}(\frac{\partial f}{\partial t})\rangle_{\theta_0}dV_{\theta_0}\\
&=-n\int_M(R_\theta-\overline{R}_\theta)|\nabla_{\theta} f|^2_{\theta}dV_{\theta}
-2\int_M\frac{\partial f}{\partial t}\Delta_{\theta} fdV_{\theta}\\
\end{split}
\end{equation}
where we have used integration by parts, (\ref{A2}) and (\ref{A9}).
It follows from the definition of $\mathcal{G}(\theta(t),f(t))$
in (\ref{6.8}) that
\begin{equation}
G(\theta(t_2),f(t_2))-G(\theta(t_1),f(t_1))=\int_{t_1}^{t_2}\mathcal{G}(\theta(t),f(t))dt.
\end{equation}
Since $f(t_2)$ is the corresponding eigenfunction of
$\lambda_1(t_2)$, we have
\begin{equation}
G(\theta(t_2),f(t_2))=\lambda_1(t_2)\int_{M}f(t_2)^2dV_{\theta(t_2)}=\lambda_1(t_2)
\end{equation}
by (\ref{Acondition}). On the other hand, it follows from (\ref{Acondition}) and the definition
of $G(\theta(t),f(t))$ that
\begin{equation}\label{6.10}
G(\theta(t_1),f(t_1))\geq \lambda_1(t_1)\int_{\partial M}f(t_1)^2dV_{\theta(t_1)}=\lambda_1(t_1).
\end{equation}
Now the result follows from (\ref{6.8})-(\ref{6.10}).
\end{proof}

Hence, we have the following:

\begin{prop}\label{prop6.2}
The first nonzero eigenvalue $\lambda_1$ of the sub-Laplacian along the CR Yamabe flow
$(\ref{A1})$ satisfies
$$\frac{d}{dt}\log\lambda_1\geq-n\Big(\max_MR_\theta-\overline{R}_\theta\Big)
+(n+1)\Big(\min_MR_\theta-\overline{R}_\theta\Big).$$
Here the derivative on the left hand side
is in the
sense of the $\liminf$ of backward difference quotients.
\end{prop}
\begin{proof}
Differentiate the first equation in (\ref{Acondition}) with respect to $t$, we have
\begin{equation}\label{A7}
\int_Mf\frac{\partial f}{\partial t}dV_\theta=\frac{n+1}{2}\int_Mf^2(R_\theta-\overline{R}_\theta)dV_\theta
\end{equation}
by (\ref{A3}). On the other hand,
since $f(t_2)$ is the corresponding eigenfunction of
$\lambda_1(t_2)$, we have
\begin{equation}\label{A8}
-\int_{M}\frac{\partial f(t_2)}{\partial t}\Delta_{\theta(t_2)} f(t_2)dV_{\theta(t_2)}
=\lambda_1(t_2)\int_Mf(t_2)\frac{\partial f(t_2)}{\partial t}dV_{\theta(t_2)}.
\end{equation}
Combining (\ref{A7}) and (\ref{A8}), we get
\begin{equation*}
\begin{split}
-\int_{M}\frac{\partial f(t_2)}{\partial t}\Delta_{\theta(t_2)} f(t_2)dV_{\theta(t_2)}
&=\frac{n+1}{2}\lambda_1(t_2)\int_Mf(t_2)^2(R_{\theta(t_2)}-\overline{R}_{\theta(t_2)})dV_{\theta(t_2)}\\
&\geq\frac{n+1}{2}\lambda_1(t_2)\left(\min_MR_{\theta(t_2)}-\overline{R}_{\theta(t_2)}\right),
\end{split}
\end{equation*}
where we have used (\ref{Acondition}) in the last inequality.
Since $\lambda_1(t_2)>0$, we have  for any $\epsilon>0$
\begin{equation*}
-\int_{M}\frac{\partial f(t_2)}{\partial t}\Delta_{\theta(t_2)} f(t_2)dV_{\theta(t_2)}>\frac{n+1}{2}\lambda_1(t_2)\left(\min_MR_{\theta(t_2)}-\overline{R}_{\theta(t_2)}-\epsilon\right).
\end{equation*}
Hence, by continuity, we can conclude that for any $\epsilon>0$
\begin{equation}\label{A10}
-\int_{M}\frac{\partial f}{\partial t}\Delta_{\theta} f(t_2)dV_{\theta}
\geq\frac{n+1}{2}\lambda_1(t_2)\left(\min_MR_{\theta}-\overline{R}_{\theta}-\epsilon\right)
\end{equation}
when  $t$ is sufficiently closed to $t_2$.
On the other hand,
\begin{equation*}
\begin{split}
&\int_M(R_{\theta(t_2)}-\overline{R}_{\theta(t_2)})|\nabla_{\theta(t_2)} f|^2_{\theta(t_2)}dV_{\theta(t_2)}\\
&\geq-\Big(\max_M R_{\theta(t_2)}-\overline{R}_{\theta(t_2)}\Big)\int_M|\nabla_{\theta(t_2)} f|^2_{\theta(t_2)}dV_{\theta(t_2)}=-\lambda_1(t_2)\Big(\max_MR_{\theta(t_2)}-\overline{R}_{\theta(t_2)}\Big)
\end{split}
\end{equation*}
by (\ref{Acondition}) and the fact that $f(t_2)$ is the corresponding eigenfunction of
$\lambda_1(t_2)$.
Since $\lambda_1(t_2)>0$, we have  for any $\epsilon>0$
\begin{equation*}
\int_M(R_{\theta(t_2)}-\overline{R}_{\theta(t_2)})|\nabla_{\theta(t_2)} f|^2_{\theta(t_2)}dV_{\theta(t_2)}
>-\lambda_1(t_2)\Big(\max_MR_{\theta(t_2)}-\overline{R}_{\theta(t_2)}+\epsilon\Big).
\end{equation*}
Hence, by continuity,
we can conclude that for any $\epsilon>0$
\begin{equation}\label{A11}
\begin{split}
\int_M(R_\theta-\overline{R}_\theta)|\nabla_{\theta} f|^2_{\theta}dV_{\theta}
&\geq-\lambda_1(t_2)\Big(\max_MR_\theta-\overline{R}_\theta+\epsilon\Big)
\end{split}
\end{equation}
when  $t$ is sufficiently closed to $t_2$.
Substituting (\ref{A10}) and (\ref{A11}) into
the inequality in Lemma \ref{AAUX}, we obtain
\begin{equation*}
\begin{split}
&\lambda_1(t_2)-\lambda_1(t_1)\\
&\geq
-n\lambda_1(t_2)\int_{t_1}^{t_2}\Big(\max_MR_\theta-\overline{R}_\theta+\epsilon\Big)dt+
(n+1)\lambda_1(t_2)
\int_{t_1}^{t_2}\Big(\min_MR_\theta-\overline{R}_\theta-\epsilon\Big)dt
\end{split}
\end{equation*}
for $t_1<t_2$ and $t_1$ sufficiently closed to $t_2$.
Now the assertion follows from dividing the last inequality by $t_2-t_1$,
letting $t_1$ go to $t_2$, and then letting $\epsilon\to 0$, as we have done in
the proof of Proposition \ref{prop2}.
\end{proof}

Similar to Proposition \ref{prop6.2},
we can  prove the following:
\begin{prop}\label{prop6.3}
The first nonzero eigenvalue $\lambda_1$ of the sub-Laplacian along the CR Yamabe flow
$(\ref{A1})$ satisfies
$$\frac{d}{dt}\log\lambda_1\leq-n\Big(\min_MR_\theta-\overline{R}_\theta\Big)
+(n+1)\Big(\max_MR_\theta-\overline{R}_\theta\Big).$$
Here the derivative on the left hand side
is in the
sense of the $\limsup$ of forward difference quotients.
\end{prop}

To prove Proposition \ref{prop6.3}, we first prove the corresponding version
of Lemma \ref{AUX1} for the CR Yamabe flow. Then we can prove Proposition \ref{prop6.3}
as we have done in the proof of Proposition \ref{prop21}.  We omit the proof and leave it to the readers.

We will omit the proof of the following proposition, since the idea is the same as
the proof of Proposition \ref{prop4}-\ref{prop6}. More precisely,
one can use the fact that
$t\mapsto \overline{R}_{\theta}$ is nonincreasing along the CR Yamebe flow
(c.f. see Proposition 3.3 in \cite{Ho2}) and apply the maximum principle to \eqref{A4}
to prove the following:

\begin{prop}\label{propA4}
If $\displaystyle\max_MR_{\theta_0}<0$, then
$$\min_MR_{\theta_0}\leq\min_MR_\theta\leq\max_MR_{\theta}\leq \max_MR_{\theta_0}<0\mbox{ for all }t\geq 0$$
under the CR Yamabe flow $(\ref{A1})$.
\end{prop}

We will also omit the proof of the following lemma, which is to apply the
the maximum principle to \eqref{A4} and is the same as
the proof of Lemma \ref{thm1} and Lemma \ref{thm2}.

\begin{lem}\label{thm3}
If $\displaystyle\max_MR_{\theta_0}<0$, then
\begin{equation*}
R_{\theta(t)}\leq \overline{R}_{\theta_0}+\left(\max_MR_{\theta_0}-\min_MR_{\theta_0}\right)
+\left(\max_MR_{\theta_0}\right)\int_0^{t}\left(\max_MR_{\theta(s)}-\overline{R}_{\theta(s)}\right)ds.
\end{equation*}
and
\begin{equation*}
R_{\theta(t)}\geq\overline{R}_{\theta_0}-\left(\max_MR_{\theta_0}-\min_MR_{\theta_0}\right)
+\left(\max_MR_{\theta_0}\right)\int_0^{t}\left(\min_MR_{\theta(s)}-\overline{R}_{\theta(s)}\right)ds.
\end{equation*}
for all $t\geq 0$ under the CR Yamabe flow $(\ref{A1})$
\end{lem}

Given a contact form $\theta$, we let $\lambda_1(\theta)$ be the first nonzero eigenvalue of
the sub-Laplacian of $\theta$. We have the following:

\begin{theorem}\label{mainA}
Suppose $(M,\theta_0)$ is a compact strictly pseudoconvex manifold of real dimension $2n+1$ such that
$\displaystyle\max_MR_{\theta_0}< 0$,
and  $\theta_Y$ is the contact form conformal to $\theta_0$
such that its Webster scalar curvature is constant and
\begin{equation}\label{Assumption}
\int_MdV_{\theta_Y}=\int_MdV_{\theta_0}.
\end{equation}
Then we have
\begin{equation}\label{A20}
e^{-c}\lambda_1(\theta_Y)\leq\lambda_1(\theta_0)\leq
e^{c}\lambda_1(\theta_Y)
\end{equation}
where $\displaystyle c=2(2n+1)\left(\frac{\min_MR_{\theta_0}}{\max_MR_{\theta_0}}-1\right).$
\end{theorem}
\begin{proof}
It was proved by Zhang in \cite{Zhang} that $\theta\to \theta_\infty$ as $t\to\infty$
under the CR Yamabe flow (\ref{A1}) such that $\theta_\infty$ has constant Webster scalar curvature.
Along the CR Yamabe flow (\ref{A1}), we have
$$\frac{d}{dt}\left(\int_MdV_\theta\right)=\int_M\frac{\partial}{\partial t}(dV_\theta)=-(n+1)\int_M(R_\theta-\overline{R}_\theta)dV_\theta=0$$
by (\ref{A0}) and (\ref{A3}),
which implies that
$\displaystyle\int_MdV_\theta=\int_MdV_{\theta_0}$ for all $t\geq 0$.
In particular, we have
\begin{equation}\label{A20a}
\int_MdV_{\theta_\infty}=\int_MdV_{\theta_0}.
\end{equation}
On the other hand, note that  $R_{\theta_Y}=c^{\frac{2}{n}}R_{\theta_\infty}$
for some constant $c>0$. Indeed, one can take $c$ to be
$\displaystyle\left(R_{\theta_Y}/R_{\theta_\infty}\right)^{\frac{n}{2}}$.
This  implies that the metric
$c^{\frac{2}{n}}\theta_Y$ has scalar curvature being equal to
$$R_{c^{\frac{2}{n}}\theta_Y}=c^{-\frac{2}{n}}R_{\theta_Y}=R_{\theta_\infty}.$$
Hence, we can conclude that
\begin{equation}\label{A20b}
c^{\frac{2}{n}}\theta_Y=\theta_\infty
\end{equation}
using the result of Jerison and Lee (see Theorem 7.1 in \cite{Jerison&Lee3} and also Theorem 1.3 in \cite{Ho1}),
which says that
if $\theta_1$ and $\theta_2$ are two contact forms conformal to $\theta_0$ such that
their Webster scalar curvatures satisfy $R_{\theta_1}=R_{\theta_2}<0$, then
$\theta_1=\theta_2$. Therefore, by (\ref{A20a}) and (\ref{A20b}), we have
$$\int_MdV_{\theta_0}=\int_MdV_{\theta_\infty}=\int_MdV_{c^{\frac{2}{n}}\theta_Y}=c^{\frac{2n+2}{n}}\int_MdV_{\theta_Y}=
c^{\frac{2n+2}{n}}\int_MdV_{\theta_0}$$
where the last equality follows from (\ref{Assumption}).
This implies that $c=1$, or equivalently,
\begin{equation}\label{A20c}
\theta_Y=\theta_\infty.
\end{equation}

Note that
\begin{equation}\label{A20d}
\min_M R_{\theta_0}\leq\overline{R}_{\theta(t)}\leq\max_{M}R_{\theta_0} \mbox{ for all }t\geq 0
\end{equation}
by Proposition \ref{propA4}.
Therefore, it follows from (\ref{A20d}) and the first inequality of Lemma \ref{thm3} that
\begin{equation}\label{A18}
\begin{split}
&\left(\max_MR_{\theta_0}\right)\int_0^{t}\left(\max_MR_{\theta(s)}-\overline{R}_{\theta(s)}\right)ds\\
&\geq \left(\overline{R}_{\theta(t)}-\overline{R}_{\theta_0}\right)+
\left(\max_MR_{\theta(t)}-\overline{R}_{\theta(t)}\right)-\left(\max_MR_{\theta_0}-\min_MR_{\theta_0}\right)\\
&\geq\left(\max_MR_{\theta(t)}-\overline{R}_{\theta(t)}\right)-2\left(\max_MR_{\theta_0}-\min_MR_{\theta_0}\right),
\end{split}
\end{equation}
and it follows from (\ref{A20d}) and the second inequality of Lemma \ref{thm3} that
\begin{equation}\label{A18a}
\begin{split}
&-\left(\max_MR_{\theta_0}\right)\int_0^{t}\left(\min_MR_{\theta(s)}-\overline{R}_{\theta(s)}\right)ds\\
&\geq \left(\overline{R}_{\theta_0}-\overline{R}_{\theta(t)}\right)+\left(\overline{R}_{\theta(t)}-\min_MR_{\theta(t)}\right)-\left(\max_MR_{\theta_0}-\min_MR_{\theta_0}\right)\\
&\geq
\left(\overline{R}_{\theta(t)}-\min_MR_{\theta(t)}\right)-2\left(\max_MR_{\theta_0}-\min_MR_{\theta_0}\right).
\end{split}
\end{equation}
Since  $\theta(t)\to \theta_\infty$ as $t\to\infty$
by Zhang's result stated above,
by letting $t$ go to infinity,
we obtain from (\ref{A18}) and (\ref{A18a}) respectively that
\begin{equation}\label{A19}
-2\left(1-\frac{\min_MR_{\theta_0}}{\max_MR_{\theta_0}}\right)\geq \int_0^{\infty}\left(\max_MR_{\theta(s)}-\overline{R}_{\theta(s)}\right)ds
\end{equation}
and
\begin{equation}\label{A19a}
-2\left(1-\frac{\min_MR_{\theta_0}}{\max_MR_{\theta_0}}\right)\geq -\int_0^{\infty}\left(\min_MR_{\theta(s)}-\overline{R}_{\theta(s)}\right)ds.
\end{equation}
Integrating the inequality in Proposition \ref{prop6.3}
and using (\ref{A20c}), (\ref{A19}) and (\ref{A19a}), we get
\begin{equation*}
\begin{split}
\log\frac{\lambda_1(\theta_Y)}{\lambda_1(\theta_0)}&=
\log\frac{\lambda_1(\theta_\infty)}{\lambda_1(\theta_0)}\\
&\leq -n\int_0^\infty\Big(\min_MR_{\theta(s)}-\overline{R}_{\theta(s)}\Big)ds
+(n+1)\int_0^\infty\Big(\max_MR_{\theta(s)}-\overline{R}_{\theta(s)}\Big)ds\\
&\leq
-2(2n+1)\left(1-\frac{\min_MR_{\theta_0}}{\max_MR_{\theta_0}}\right),
\end{split}
\end{equation*}
which gives the lower bound for $\lambda_1(\theta_0)$ in (\ref{A20}).
On the other hand, integrating the  inequality in Proposition \ref{prop6.2}
and using (\ref{A20c}), (\ref{A19}) and (\ref{A19a}), we get
\begin{equation*}
\begin{split}
\log\frac{\lambda_1(\theta_Y)}{\lambda_1(\theta_0)}&=
\log\frac{\lambda_1(\theta_\infty)}{\lambda_1(\theta_0)}\\
&\geq -n\int_0^\infty\Big(\max_MR_{\theta(s)}-\overline{R}_{\theta(s)}\Big)ds
+(n+1)\int_0^\infty\Big(\min_MR_{\theta(s)}-\overline{R}_{\theta(s)}\Big)ds\\
&\geq 2(2n+1)\left(1-\frac{\min_MR_{\theta_0}}{\max_MR_{\theta_0}}\right),
\end{split}
\end{equation*}
which gives the upper bound for $\lambda_1(\theta_0)$ in (\ref{A20}). This proves the assertion.
\end{proof}

\section{The unnormalized CR Yamabe flow}

In this section, we consider the unnormalized CR Yamabe flow:
\begin{equation}\label{MA1}
\frac{\partial}{\partial t}\theta=-R_\theta\,\theta\mbox{ for }t\geq 0,\hspace{2mm}\theta|_{t=0}=\theta_0.
\end{equation}
If we write $\theta=u^{\frac{2}{n}}\theta_0$, then
\begin{equation}\label{MA2}
\frac{\partial u}{\partial t}=-\frac{n}{2}R_\theta u\mbox{ for }t\geq 0,\hspace{2mm}u|_{t=0}=1.
\end{equation}
Hence, the volume form $dV_\theta$ and the Webster scalar curvature $R_\theta$ of $\theta$ satisfy (see (6.3) in \cite{Ho1})
\begin{eqnarray}\label{MA3}
&&\frac{\partial}{\partial t}(dV_\theta)=-(n+1)R_\theta\,dV_\theta,\\
\label{MA4}
&&\frac{\partial}{\partial t}R_\theta=(n+1)\Delta_\theta R_\theta+R_\theta^2.
\end{eqnarray}

Let $\lambda_1$ be the first eigenvalue of $-\Delta_\theta+aR_\theta$ where $a$ is a constant, i.e.
\begin{equation}\label{MA5}
-\Delta_\theta f+aR_\theta f=\lambda_1 f
\end{equation}
for some function $f$.
Again we assume that
there is a family of the first eigenvalue and the corresponding eigenfunction which is $C^1$ in $t$.
By rescaling, we may assume that the eigenfunction $f$ satisfies
\begin{equation}\label{MA6}
\int_Mf^2dV_\theta=1.
\end{equation}

\begin{prop}\label{propA1}
Along the unnormalized CR Yamabe flow \eqref{A1}, we have
\begin{equation*}
\frac{d\lambda_1}{dt}=\left(2(n+1)a-n\right)\int_MR_\theta(|\nabla_{\theta} f|^2_{\theta}+aR_\theta f^2)dV_{\theta}
-(n+1)(2a-1)\lambda_1\int_MR_\theta f^2dV_\theta.
\end{equation*}
\end{prop}
\begin{proof}
Differentiate (\ref{MA6}) with respect to $t$, we have
\begin{equation}\label{MA7}
\int_Mf\frac{\partial f}{\partial t}dV_\theta=\frac{n+1}{2}\int_MR_\theta f^2dV_\theta
\end{equation}
by (\ref{MA3}).
Multiply (\ref{MA5}) by $f$ and integrate it over $M$, we obtain
\begin{equation}\label{MA8}
\lambda_1=\int_M(|\nabla_\theta f|^2_\theta+aR_\theta f^2)dV_\theta
\end{equation}
by (\ref{MA6}) and integration by parts. Since $\theta=u^{\frac{2}{n}}\theta_0$,
\eqref{A9} holds.
Combining \eqref{A9} and (\ref{MA8}), we have
$$\lambda_1=\int_Mu^2|\nabla_{\theta_0} f|^2_{\theta_0}dV_{\theta_0}+a\int_MR_\theta f^2dV_\theta.$$
Differentiate it with respect to $t$, we obtain
\begin{equation*}
\begin{split}
\frac{d\lambda_1}{dt}&=\int_M2u\frac{\partial u}{\partial t}|\nabla_{\theta_0} f|^2_{\theta_0}dV_{\theta_0}
+2\int_Mu^2\langle\nabla_{\theta_0} f,\nabla_{\theta_0}(\frac{\partial f}{\partial t})\rangle_{\theta_0}dV_{\theta_0}\\
&\hspace{4mm}+a\int_Mf^2\frac{\partial R_\theta}{\partial t}dV_\theta+2a\int_MR_\theta f\frac{\partial f}{\partial t}dV_\theta
+a\int_MR_\theta f^2\frac{\partial}{\partial t}(dV_\theta)\\
&=-n\int_MR_\theta|\nabla_{g} f|^2_{\theta}dV_{\theta}
+2\int_M\langle\nabla_{\theta} f,\nabla_{\theta}(\frac{\partial f}{\partial t})\rangle_{\theta}dV_{\theta}\\
&\hspace{4mm}+a\int_Mf^2\big((n+1)\Delta_\theta R_\theta+R_\theta^2\big)dV_\theta+2a\int_MR_\theta f\frac{\partial f}{\partial t}dV_\theta
-(n+1)a\int_MR_\theta^2f^2dV_\theta\\
&=-n\int_MR_\theta|\nabla_{\theta} f|^2_{\theta}dV_{\theta}
+2\int_M\frac{\partial f}{\partial t}\big(-\Delta_\theta f+aR_\theta f\big)dV_\theta\\
&\hspace{4mm}+(n+1)a\int_MR_\theta \Delta_\theta(f^2)dV_\theta
-na\int_MR_\theta^2f^2dV_\theta\\
&=\left(2(n+1)a-n\right)\int_MR_\theta|\nabla_{\theta} f|^2_{g}dV_{\theta}
+2\lambda_1\int_Mf\frac{\partial f}{\partial t}dV_\theta\\
&\hspace{4mm}-2(n+1)a\int_MR_\theta f\big(-\Delta_\theta f+aR_\theta f\big)dV_\theta
+\left(2(n+1)a^2-na\right)\int_MR_\theta^2f^2dV_\theta\\
&=\left(2(n+1)a-n\right)\int_MR_\theta(|\nabla_{\theta} f|^2_{\theta}+aR_\theta f^2)dV_{\theta}
-(n+1)(2a-1)\lambda_1\int_MR_\theta f^2dV_\theta
\end{split}
\end{equation*}
where the second equality follows from (\ref{A9}) and (\ref{MA2})-(\ref{MA4}),
the third equality follows from integration by parts, and the last two equalities follow from (\ref{MA5}) and (\ref{MA7}).
This proves the assertion.
\end{proof}

Now, by the same proof of Proposition \ref{propC2}-\ref{propC3}, we have the following:

\begin{prop}\label{MpropA4}
Along the unnormalized CR Yamabe flow \eqref{MA1},
\\
(i)
if $0\leq a\leq\displaystyle\frac{n}{2n+2}$ and
$\displaystyle\min_MR_\theta\geq\frac{n}{n+1}\max_MR_\theta\geq 0$, then
$\displaystyle\frac{d\lambda_1}{dt}\geq 0$, and equality holds if and only if $R_\theta$ is constant;\\
(ii) if $\displaystyle \frac{n}{2n+2}\leq a\leq \frac{1}{2}$ and
$\displaystyle\min_MR_\theta\geq 0$, then
$\displaystyle\frac{d\lambda_1}{dt}\geq 0$
and equality holds if and only if $R_\theta\equiv 0$; and\\
(iii) if $a\geq\displaystyle\frac{1}{2}$ and
$\displaystyle\min_MR_\theta\geq 0$, then
$\displaystyle\frac{d\lambda_1}{dt}\geq 0$
and equality holds if and only if $R_\theta\equiv 0$.
\end{prop}

Proposition \ref{MpropA4} implies the following:

\begin{theorem}\label{Mmain2}
Along the unnormalized CR Yamabe flow \eqref{MA1},
the first eigenvalue of $-\Delta_\theta+aR_\theta$ is nondecreasing\\
(i)
if $0\leq a<\displaystyle\frac{n}{2n+2}$ and
$\displaystyle\min_MR_\theta\geq\frac{n}{n+1}\max_MR_\theta\geq 0$;\\
(ii) if $a\geq\displaystyle\frac{n}{2n+2}$ and
$\displaystyle\min_MR_\theta\geq 0$.
\end{theorem}

Note that Theorem \ref{Mmain2} was obtained in \cite{Ho2} for the cases when $a=0$ and $a=\displaystyle \frac{n}{2n+2}$.
See Theorem 1.4 and Theorem 1.5 in \cite{Ho2}. Note also that Chang-Lin-Wu \cite{Chang&Lin&Wu}
has obtained a result similar to Theorem \ref{Mmain2} for the case when $n=1$.

\bibliographystyle{amsplain}

\end{document}